\documentclass[12pt,reqno,psfig,openbib]{amsart}

\usepackage[pagebackref=true, colorlinks=true, citecolor=blue]{hyperref}

\usepackage{amssymb,amsmath,graphicx,amsfonts,euscript}
\usepackage{fancyhdr}
\usepackage{epsfig}
\usepackage{setspace}
\usepackage{dsfont}
\usepackage[all]{xy}
\usepackage{amsmath}

 \usepackage{pst-grad} 
 \usepackage{pst-plot} 

\newtheorem{theorem}{Theorem}[section]
\newtheorem{rem}[theorem]{Remark}

\newtheorem{lemma}[theorem]{Lemma}

\numberwithin{equation}{section}

\newcommand{\ov}{\overline}

\newcommand{\pa}{\partial}
\newcommand{\ep}{\varepsilon}
\newcommand{\tht}{\theta}

\newcommand{\pri}{\prime}
\newcommand{\gam}{\gamma}
\newcommand{\sig}{\sigma}
\newcommand{\kap}{\kappa}
\newcommand{\blds}{\boldsymbol}
\newcommand{\dis}{\displaystyle}

\newcommand{\abs}[1]{\left\vert#1\right\vert}
\newcommand{\BB}[1]{\ensuremath{\mathbb{#1}}}

\newcommand{\R}{\ensuremath{\BB{R}}}

\newcommand{\iny}{\ensuremath{\infty}}
\newcommand{\grad}{\ensuremath{\nabla}}
\DeclareMathOperator{\dv}{div} %
\DeclareMathOperator{\curl}{curl} %
\newcommand{\prt}{\ensuremath{\partial}}
\newcommand{\brac}[1]{\ensuremath{\left[ #1 \right]}}
\newcommand{\pr}[1]{\ensuremath{\left( #1 \right) }}
\newcommand{\set}[1]{\ensuremath{\left\{ #1 \right\}}}

\newcommand{\norm}[1]{\ensuremath{\left\Vert #1 \right\Vert}}
\newcommand{\smallnorm}[1]{\ensuremath{\Vert #1 \Vert}}
\newcommand{\refA}[1]{Appendix~\ref{A:#1}}
\newcommand{\refS}[1]{Section~\ref{S:#1}}
\newcommand{\refSAnd}[2]{Sections~\ref{S:#1} and \ref{S:#2}}
\newcommand{\refT}[1]{Theorem~\ref{T:#1}}
\newcommand{\refTAnd}[2]{Theorems~\ref{T:#1} and \ref{T:#2}}
\newcommand{\refTAndAnd}[3]{Theorems~\ref{T:#1}, \ref{T:#2}, and \ref{T:#3}}

\newcommand{\refL}[1]{Lemma~\ref{L:#1}}

\newcommand{\refD}[1]{Definition~\ref{D:#1}}

\newcommand{\refE}[1]{(\ref{e:#1})}
\newcommand{\refESub}[2]{$(\ref{e:#1})_{#2}$}
\newcommand{\refEAnd}[2]{(\ref{e:#1}, \ref{e:#2})}
\newcommand{\refEAndAnd}[3]{(\ref{e:#1}, \ref{e:#2}, \ref{e:#3})}
\newcommand{\refEThrough}[2]{(\ref{e:#1}) through (\ref{e:#2})}

\newcommand{\refEFirst}[1]{(\ref{e:#1}}
\newcommand{\refENext}[1]{, \ref{e:#1}}
\newcommand{\refELast}[1]{, \ref{e:#1})}
\newcommand{\refEFirstSub}[2]{($\ref{e:#1}_{#2}$}

\newcommand{\refELastSub}[2]{, $\ref{e:#1}_{#2}$)}

\newcommand{\refR}[1]{Remark~\ref{R:#1}}
\newcommand{\eps}{\ensuremath{\varepsilon}}
\newcommand{\Cal}[1]{\ensuremath{\mathcal{#1}}}
\newcommand{\al}{\ensuremath{\alpha}}

\newcommand{\pdx}[2]{\frac{\prt #1}{\prt #2}}

\newcommand{\ol}{\overline}

\newcommand{\Ignore}[1] {}

\theoremstyle{definition}
\newtheorem{definition}[theorem]{Definition}
\newtheorem{remark}[theorem]{Remark}

\newcommand{\spacer}{\vspace{2mm}}

\newcommand{\normmark}{\Vert}

\newcommand{\g}{\blds{g}}
\newcommand{\n}{\blds{n}}

\subjclass[2000]{35B25, 35C20, 76D05, 76D10}
\keywords{boundary layers, singular perturbations, Navier-Stokes equations, Euler equations, Navier friction boundary condition}

\begin{document}

\title[Boundary layer with Generalized Navier boundary conditions]
{Boundary layer analysis of the Navier-Stokes equations with generalized Navier boundary conditions}
\author[G. Gie and J. Kelliher]{Gung-Min Gie$^{1}$ and James P. Kelliher$^{1}$}
\address{$^1$ Department of Mathematics, University of California, Riverside, 900 University Ave., Riverside, CA 92521, U.S.A.}
\email{gungmin@ucr.edu}
\email{kelliher@math.ucr.edu}

\begin{abstract}
We study the weak boundary layer phenomenon of the Navier-Stokes equations with generalized Navier friction boundary conditions, $u\cdot \n = 0$, $\brac{\mathbf{S}(u)  \n}_{\text{tan}} + \mathcal{A} u= 0$,
 in a bounded domain in $\mathbb{R}^3$ when the viscosity, $\eps > 0$, is small. Here, $\Cal{A}$ is a type $(1, 1)$ tensor on the boundary: when $\Cal{A} = \al I$ we obtain Navier boundary conditions, and when $\Cal{A}$ is the shape operator we obtain the conditions, $u \cdot \blds{n} = (\curl u) \times \blds{n} = 0$. By constructing an explicit corrector, we prove the convergence, as $\eps$ tends to zero, of the Navier-Stokes solutions to the Euler solution. We do this both in the natural energy norm with a rate of order $\eps^{3/4}$ as well as uniformly in time and space with a rate of order $\eps^{3/8 - \delta}$ near the boundary and $\eps^{3/4 - \delta'}$ in the interior, where $\delta, \delta'$ decrease to 0 as the regularity of the initial velocity increases.
 This work simplifies an earlier work of Iftimie and Sueur, \cite{IS10}, as we use a simple and explicit corrector (which is more easily implemented in numerical applications).
 It also improves a result of Masmoudi and Rousset, \cite{MasmoudiRousset2010}, who obtain convergence uniformly in time and space via a method that does not yield a convergence rate.
\end{abstract}

\date{13 May 2011, updated 12 July 2011 (compiled on \today)}

\maketitle

\vspace{-0.25in}

\tableofcontents

%
%
\section{Introduction}

\noindent The flow of an incompressible, constant-density, constant-viscosity Newtonian fluid is described by the Navier-Stokes equations,
\begin{align}\label{e:0-1}
	\left\{
	\begin{array}{rl}
		\dfrac{\pa u^{\ep}}{\pa t} -\ep \Delta u^{\ep}
			+ (u^{\ep} \cdot \nabla)u^{\ep}
			+ \nabla p^{\ep} = f
				& \text{ in } \Omega \times(0,T), \\
		\text{div } u^{\ep} = 0
			& \text{ in } \Omega \times(0,T), \\
		u^{\ep}|_{t=0} = u_{0}
			& \text{ in } \Omega.
	\end{array}
	\right.
\end{align}
The fluid is contained in the three-dimensional bounded domain, $\Omega$, with smooth boundary, $\Gamma$.
The parameter, $\ep > 0$ is the viscosity and $T > 0$ is fixed (see \refT{EulerWellPosedness}). The equations are to be solved for the velocity of the fluid, $u^\eps$, and pressure, $p^\eps$, given the forcing function, $f$, and initial velocity, $u_0$. The regularity of $\Gamma$, $f$, and $u_{0}$ we assume is specified in \refE{0-6}, but our emphasis is not on optimal regularity requirements.

When $\eps = 0$, we formally obtain the Euler equations,
\begin{align}\label{e:Euler}
	\left\{
	\begin{array}{rl}
		\dfrac{\pa u^0}{\pa t}  + (u^0 \cdot \nabla)u^0 + \nabla p^{0} = f
			& \text{ in } \Omega\times(0,T), \\
		\text{div }u^0 = 0
			& \text{ in } \Omega \times(0,T), \\
		u^{0}\cdot \n = 0
			& \text{ on } \Gamma \times (0,T), \\
		u^0|_{t=0} = u_{0}
			& \text{ in } \Omega,
	\end{array}
	\right.
\end{align}
where $\n$ is the outer unit normal vector on $\Gamma$.

For the Euler equations, we use the minimal, impermeable boundary conditions, $u^0 \cdot \n = 0$, reflecting no entry or exit of fluid from the domain; being a first-order equation, these conditions suffice.
No-slip boundary conditions, $u^\eps = 0$ on $\Gamma$, are those most often prescribed for the Navier-Stokes equations. This, of necessity, leads to a discrepancy between $u^\eps$ and $u^0$ at the boundary, resulting in boundary layer effects. Prandtl \cite{P1905} was the first to make real progress on analyzing these effects, and much of a pragmatic nature has been discovered, but to this day the mathematical understanding is woefully inadequate.

In part because of these difficulties with no-slip boundary conditions, and in part because of very real physical applications, researchers have turned to other boundary conditions. Of particular interest are boundary conditions variously called Navier friction, Navier slip, or simply Navier boundary conditions (other names have been used as well).  These boundary conditions can be written as
\begin{equation}\label{e:0-2}
			u^{\ep}\cdot\n = 0, \;
		\brac{\mathbf{S}(u^{\ep})  \n
			+ \alpha u^{\ep}}_{\text{tan}} = 0 \text{ on } \Gamma,
\end{equation}
where
\begin{equation}\label{e:0-3}
	\mathbf{S}(u)
		: = \dfrac{1}{2} \big( \nabla u + (\nabla u)^{\intercal}  \big)
		= \Big( \dfrac{1}{2}  \dfrac{\pa u_j}{\pa x_i}
			+  \dfrac{1}{2} \dfrac{\pa u_i }{\pa x_j}  \Big)_{1 \leq i,j \leq 3},
			\text{ for } u = (u_1, u_2, u_3).
\end{equation}
Here $(x_1, x_2, x_3)$, (or $(x,y,z)$ in \refS{S-2}), denotes the Cartesian
coordinates of a point $\blds{x} \in \mathbb{R}^3$, $\alpha$ is the (positive or negative) friction coefficient, which is independent of $\ep$. The notation $\brac{\cdot}_{\text{tan}}$ in \refE{0-2} denotes the tangential components of a vector on $\Gamma$.

In this paper, we use the generalization of \refE{0-2},
\begin{equation}\label{e:General_NBC}
    \left\{
    \begin{array}{l}
            \spacer
            u^{\ep}\cdot \n = 0 \text{ on } \Gamma, \\
            \brac{\mathbf{S}(u^{\ep}) \n}_{\text{tan}}  + \mathcal{A} \, u^{\ep} = 0 \text{ on } \Gamma,
    \end{array}
    \right.
\end{equation}
of the Navier boundary conditions. Here, $\Cal{A}$ is a type $(1, 1)$ tensor on the boundary having at least $C^2$-regularity. In coordinates on the boundary, $\mathcal{A}$ can be written in matrix form as
$
            \mathcal{A} = \big(\alpha_{ij} \big)_{1 \leq ij \leq 2}.
$
Note that $u^\eps$ lies in the tangent plane, as does $\Cal{A} u^\eps$.

It is easy to see that when $\mathcal{A} = \al I$, the product of a function $\alpha$ on $\Gamma$ and the identity tensor, the generalized Navier boundary conditions, \refE{General_NBC}, reduce to the usual Navier friction boundary conditions, \refE{0-2}. In fact, the analysis using a general $\Cal{A}$ in place of $\al I$ is changed only slightly from using $\al I$ with $\al$ a constant (we say a bit more on this in \refR{SimpleExtension}).

The primary motivation for generalizing Navier boundary conditions in this manner is that when $\mathcal{A}$ is the shape operator (Weingarten map) on $\Gamma$, one obtains, as a special case, the boundary conditions,
\begin{align}\label{e:LionsLikeBCs}
	u^\ep \cdot \blds{n} = (\curl u^\ep) \times \blds{n} = 0,
\end{align}
as we show in \refA{LionsLikeBCs}. (This fact is implicit in \cite{daViegaCrispo2011A}.) Such boundary conditions have been studied (in 3D) by several authors, including \cite{daViegaCrispo2010, daViegaCrispo2011A, XiaoXin2007} (and see the references therein), \cite{BelloutNeustupa1, BelloutNeustupa2} for an inhomogeneous version of \refE{LionsLikeBCs}, and \cite{BelloutNeustupa3, BelloutNeustupa4} for related boundary conditions. In this special case, stronger convergence can be obtained (at least in a channel), in large part because vorticity can be controlled near the boundary. Hence, somewhat different issues arise, and the bodies of literature studying boundary conditions \refE{LionsLikeBCs} and \refE{0-2} are somewhat disjoint.


We  introduce the Hilbert space,
\begin{align*} 
	H = \set{u \in L^2(\Omega)^3 \colon \text{ div }u =0 \text{ and } u \cdot \n = 0 \text{ on } \Gamma},
\end{align*}
equipped with the usual $L^2$ inner product. Then, letting $T>0$ be an arbitrary time less than any $T$ appearing in \refTAnd{EulerWellPosedness}{MR}, we state our main result:
\begin{theorem}\label{T:1}
Let $\Gamma_a$ be the interior tubular neighborhood of $\Omega$ with width $a > 0$. Assume that
\begin{align}\label{e:0-6}
	u_0 \in H \cap H^{m}(\Omega),
		\; f \in C^\iny_{loc}([0, \iny); C^\iny(\Omega)),
		\; \Gamma \text{ is } C^{m+2}
		\text{ for } m \ge 5.
\end{align}
Then $u^\eps$, a solution of the Navier-Stokes equations, \refE{0-1}, with Generalized Navier boundary conditions, \refE{General_NBC}, converges to $u^0$, the solution of the Euler equations, \refE{Euler}, as the viscous parameter $\ep$ tends to zero, in the sense that
\begin{align}\label{e:68}
	\norm{u^{\ep}  -  u^{0}}_{L^{\infty}(0, T; L^{2}(\Omega))}
		\le \kap \ep^{\frac{3}{4}},
	\quad
	\norm{u^{\ep} - u^{0}}_{L^{2}(0, T ; H^{1}(\Omega))}
		\leq\kap \ep^{\frac{1}{4}},
\end{align}
for some $T > 0$ and for a constant $\kap = \kap (T, \ov{\alpha}, u_0, f)$, $\ol{\al} = \norm{\Cal{A}}_{C^m(\Gamma)}$. If $m > 6$ and $f \equiv 0$ then
	\begin{align}\label{e:uniform}
		\begin{split}
		\norm{u^\eps - u^0}_{L^\iny([0, T] \times \Gamma_a)}
			&\le \kappa
				\eps^{\frac{3}{8} - \frac{3}{8(m - 1)}},
		\quad
		\norm{u^\eps - u^0}_{L^\iny([0, T] \times \Omega \setminus \Gamma_a)}
			\le \kappa
				\eps^{\frac{3}{4} - \frac{9}{8m}},
		\end{split}
	\end{align}
where now $\kap = \kap(T, \ov{\alpha}, m, a, u_0, f)$.
\end{theorem}

	Because we will only have existence of $u^\eps$ when
	$5 \le m \le 6$ (\refT{NSExistence}), by $u^\eps$ we mean
	an arbitrary choice of the possibly multiple solutions when we consider the limit
	as $\eps \to 0$. When $m > 6$
	the solutions are unique
	as shown by Masmoudi and Rousset (see \refT{MR}), and
	this arbitrary choice becomes unnecessary.

\begin{remark}
	Standard boundary layer analysis indicates that a linear corrector will be of order $\eps^{1/2}$ in $L^\iny([0, T] \times \Omega)$,
	so an exponent of $\frac{1}{2}$ rather than $\frac{3}{8}$ in \refE{uniform} should be considered
	optimal (for $C^\iny$ initial data).
\end{remark}

Navier boundary conditions go back to \cite{N1827}, in which Navier first proposed them, and to \cite{M1879}, in which Maxwell derived them from the kinetic theory of gases. There has been intermittent interest in them since, but revival of active interest in the mathematical community working on the vanishing viscosity limit started with the paper of Clopeau, Mikeli{\'c}, and Robert \cite{CMR}, which gives a vanishing viscosity result in two dimensions. Also, the work of J-M Coron in \cite{Coron1995} on the controllability of the 2D Navier-Stokes equations with Navier boundary conditions, which precedes \cite{CMR}, initiated interest in these boundary conditions in the PDE control theory community. By now there is a fairly substantial mathematical literature on the subject,
but the three papers, \cite{IP06, IS10, MasmoudiRousset2010}, are of particular concern to us here.
Both \cite{IP06} and \cite{MasmoudiRousset2010} give existence theorems for solutions to \refEAnd{0-1}{0-2}, with uniqueness holding for stronger initial data. We quote these results in \refTAnd{NSExistence}{MR}.

Even with Navier boundary conditions there is a discrepancy between $u^0$ and $u^\eps$ on the boundary, so we expect boundary layer effects to occur. As first shown (in 3D) by Iftimie and Planas in \cite{IP06}, however, this boundary layer effect is mild enough to allow convergence of $u^{\ep}$ to $u^{0}$ in $L^{\infty}(0,T; L^2{(\Omega)})$ without using any artificial function correcting the difference $u^{\ep} - u^0$ on the boundary. (This result of \cite{IP06} was for $\al \ge 0$, but the argument is easily modified to allow $\al$ to be negative.)  Thus, it makes sense to refer to the boundary layer as weak.

Specifically, Iftimie and Planas show in \cite{IP06} that
\begin{equation}\label{e:c9}
	\normmark  u^{\ep} - u^0 \normmark _{L^{\infty}(0,T; L^2(\Omega))}
		\leq
		C \ep^{\frac{1}{2}}.
\end{equation}
Iftimie and Sueur \cite{IS10} use a corrector, $\widetilde{v}$, to improve the convergence rate in \refE{c9} to $\eps^{3/4}$ in this energy norm.
More precisely, they consider an asymptotic expansion of $u^{\ep}$ as the sum of $u^0$ and $\widetilde{v}$, where $\widetilde{v}$ is a corrector whose main part of its tangential components is defined as a solution of a linearized Prandtl-type equation. Using the properties of $\widetilde{v}$, they show that
\begin{align*} 
u^{\ep} - (u^0 + \widetilde{v}) \text{ is order } \ep \text{ in } L^{\infty}(0,T; L^2(\Omega)),
	\text{ order } \ep^{\frac{1}{2}} \text{ in } L^{2}(0,T; H^1(\Omega)).
\end{align*}
These bounds with estimates on the corrector $\widetilde{v}$ then give
\begin{align}\label{e:uu0Conv}
	\norm{u^{\ep}  -  u^{0}}_{L^{\infty}(0, T; L^{2}(\Omega))}
		\le C \ep^{\frac{3}{4}},
	\quad
	\norm{u^{\ep} - u^{0}}_{L^{2}(0, T ; H^{1}(\Omega))}
		\le C \ep^{\frac{1}{4}}.
\end{align}

We, on the other hand, use an asymptotic expansion of $u^{\ep}$ in the form
$
u^{\ep} \simeq u^0 + \tht^{\ep},
$
where the main part of the explicitly defined corrector, $\tht^{\ep}$, exponentially decays from the boundary; see \refEAnd{24}{25} and \refEAnd{226}{227}. Because our corrector is so simple and explicit it can be mored readily used in numerical applications than that of \cite{IS10}, which requires the solution of a coupled system of linear equations. Both correctors are linear and both can be used to obtain order $\eps^{3/4}$ convergence in the vanishing viscosity limit, \refE{68}, but Iftimie and Sueur achieve an order of convergence of $\eps$ for the \textit{corrected} difference, $u^\eps - u^0 - \widetilde{v}$, while our corrected difference still gives order $\eps^{3/4}$. The tradeoff is simplicity of the corrector versus rate of convergence of the corrected velocity.

\Ignore{ 
By using a simpler corrector, we lose some accuracy, with respect to $\ep$, in the error analysis. That is, replacing $ \widetilde{v}$ by $\tht^{\ep}$, the estimates \refE{IS1} get worse by $\ep^{1/4}$. However, the main convergence results \refE{uu0Conv} still hold true because the size of $\tht^{\ep}$ with respect to $\ep$ in certain key Sobolev spaces is the same as that of $\widetilde{v}$. In this sense, our simple corrector function explains the  essential behavior near the boundary of the flow governed by the Navier-Stokes equations with the Navier friction boundary conditions. Furthermore, thanks to its simple and explicit expression, $\tht^{\ep}$ can be mored readily used in numerical applications related to this problem.
} 

We wish to emphasize that the techniques we employ in this paper differ considerably from those of \cite{IS10}. While the approach in both papers originates in the work of Prandtl \cite{P1905}, our approach adheres much more closely to a by now well-established approach to boundary layer analysis, which we adapt to treat Navier boundary conditions. In this regard our arguments will be more familiar to many researchers, and hence, ultimately, we believe, easier to 
incorporate into the existing understanding of boundary layers as they appear in a variety of physical problems.
(For a description of the general theory of boundary layer analysis, see, for example, \cite{EW2000, Eck72, Gie1, GHT1, Hol95, Lion73, O'M77, ShK87, VL62}. Concerning boundary layer analysis related to the Navier-Stokes equations, we refer readers to \cite{GHT2, Gre98, HT1, HT2, IS10, KTW2010, TW1996OseenIUMJ, TW02-1, TW95-1, TW97-1, TW1998}.)

A key aspect of our corrector is that it is coordinate-independent. This not only gives it geometric meaning, it removes the need for a partition of unity to patch together the corrector defined in charts throughout the boundary layer. Nonetheless, the corrector has a particularly simple form in principal curvature coordinates, which we discuss in \refS{Intrinsic}. (Such coordinates are used in much the same way, though for different purposes, in \cite{daViegaCrispo2011A}.)

We also, in \refE{uniform}, obtain convergence uniformly in time and space of order $\eps^{3/8 - \delta}$ near the boundary and $\eps^{3/4 - \delta'}$ in the interior, with $\delta, \delta'$ decreasing as the regularity of the initial velocity is increased, by employing an anisotropic embedding inequality developed in \refS{Agmons}. We take great advantage of the regularity result of Masmoudi and Rousset \cite{MasmoudiRousset2010} (\refT{MR}) to obtain this convergence. The authors of \cite{MasmoudiRousset2010} themselves take a similar approach; however, the anisotropic inequality they use requires control on norms higher than those in \refE{uu0Conv}, and this is only sufficient to obtain boundedness of the sequence of solutions to \refEAnd{0-1}{0-2}. A compactness argument then gives
convergence uniformly in time and space, though without a rate of convergence.

The body of this paper is organized as follows: The existence and uniqueness results for solutions to the Navier-Stokes equations and Euler equations that we will need are given in \refS{ExistUnique}. We give the proof of \refE{68}, the first part of \refT{1}, in \refSAnd{S-2}{S-3}. To avoid the geometrical difficulties of a curved boundary, which obscure the key ingredients of the argument, we first prove \refE{68} for the case of a three-dimensional periodic channel domain. We do this in \refS{S-2}. Then, in \refS{S-3}, as a generalization of \refS{S-2}, we treat the case of a bounded domain in $\mathbb{R}^3$ with smooth (curved) boundary. In \refS{S-4}, we present the (very short) proof of \refE{uniform}, the second part of \refT{1}, which relies on the anisotropic Agmon's inequality, which we establish in \refT{AnistropicAgmonsInequality}.
In \refA{LionsLikeBCs}, we prove that \refE{General_NBC} reduces to \refE{LionsLikeBCs} when $\Cal{A}$ is the shape operator. Finally, \refS{SomeLemmas} contains some standard lemmas which we state without proof.


%
%
\section{Existence and Uniqueness Theorems}\label{S:ExistUnique}

\noindent Thanks to \refL{1}, by applying the Galerkin method, one can construct solutions to \refE{0-1} with \refE{General_NBC} in the following sense, as shown in \cite{IS10} (see \refR{SimpleExtension}):
\begin{theorem}[Iftimie, Sueur \cite{IS10}]\label{T:NSExistence}   
Assuming that $u_0$ lies in $H$ and $f$ lies in $L^{2}(0, T; L^2(\Omega)^3)$, there exists a weak solution $u^{\ep} \in \mathcal{C}^0_{w}(0, T; H) \cap L^2(0, T; H^1(\Omega)^3)$ of the Navier Stokes equations, \refE{0-1}, with the generalized Navier friction boundary conditions, \refE{General_NBC}, in the sense that, for any $v \in H \cap \mathcal{C}^{\infty}_0 ([0,T] \times \overline{\Omega})$,
\begin{align*}\begin{array}{l}
\displaystyle \spacer - \int_{0}^{T} \int_{\Omega} u^{\ep} \cdot \dfrac{\pa v}{\pa t} \, d \blds{x} \, dt + 2 \ep \int_{0}^{T}\int_{\Omega}  \mathbf{S}(u^{\ep}) \cdot \mathbf{S}(v) \, d \blds{x} \, dt \\
\displaystyle  + 2 \ep \int_{0}^{T} \int_{\Gamma} \Cal{A} u^{\ep} \cdot v \, dS \, dt  +  \int_{0}^{T}\int_{\Omega} (u^{\ep}\cdot \nabla)u^{\ep} \cdot v \, d \blds{x} \, dt
= \int_{\Omega} u_0 \cdot v|_{t=0} \, d \blds{x}.
\end{array}
\end{align*}
\end{theorem}

We have the following well-posedness result for solutions to the Euler equations:
\begin{theorem}\label{T:EulerWellPosedness}
	Suppose that $u^0$ lies in $H \cap H^m(\Omega) \cap C^{1, \mu}(\Omega)$, $\mu$ in $(0, 1]$, $m \ge 3$
	is an integer,
	$f$ lies in $C^\iny_{loc}([0, \iny); C^\iny(\ol{\Omega}))$, and $\Gamma$ is of class $C^{m+2}$.
	Then for some time $T > 0$ there exists a unique solution, $u^0$, to \refE{Euler} lying in
	$C^1_b([0,T] \times \Omega) \cap C([0, T]; H^m(\Omega))$. The corresponding pressure, $p^0$,
	lies in $L^{\infty}(0, T; H^{m+1}(\Omega))$ and is unique up to an additive function of time.
\end{theorem}
\begin{proof}
Combining Theorem 1 and Theorem 2 part 3 of \cite{Koc02} gives the existence, uniqueness, and regularity of $u$ when $f \equiv 0$ and the boundary is smooth. The proof is straightforward to adapt to smooth forcing, and the strongest restriction on the smoothness of the boundary comes through the use of the Leray projector (Lemma 2 of \cite{Koc02}), where, however, $C^{m + 2}$-regularity suffices. The regularity of the pressure (as well as the well-posedness in Sobolev spaces) is proved in \cite{TemEuler75, TemEuler76}.
\end{proof}

When \refE{0-6} holds, by virtue of \refT{EulerWellPosedness} and Sobolev embedding, for some $T > 0$,
\begin{align}\label{e:0-7}
	u^0 \in C^1_b([0,T] \times \Omega) \cap C([0, T]; H^m(\Omega)),
	\; p^0 \in L^{\infty}(0, T; H^{m+1}(\Omega))
	\text{ for } m \ge 5.
\end{align}
The regularity in \refE{0-7} of the solution is what we require; we do not claim that the assumptions in \refE{0-6} are the minimal ones that guarantee such regularity, however.

In \cite{MasmoudiRousset2010}, Masmoudi and Rousset obtain the well-posedness result that we state in \refT{MR} for solutions to  \refEAnd{0-1}{General_NBC} in the conormal Sobolev spaces of \refD{HConormal}   (see \refR{SimpleExtension}).

 \begin{definition}\label{D:HConormal}
Let $\Omega$ be a $d$-dimensional manifold, $d \ge 1$, with $C^k$-boundary, $k \ge 1$. Viewing vector fields as derivations of $C^\iny(\Omega)$, we say that a vector field, $X$, is tangent to $\prt \Omega$ if $X f = 0$ on $\prt \Omega$ whenever $f$ is constant on $\Omega$. Let $(Z_j)_{j=1}^N$ be a set of generators of vector fields tangent to $\prt \Omega$. (Locally, only $d$ such vector fields are needed, but for a global basis, $N$ will be greater than $d$.) For a multiindex, $\beta$, let $Z^\beta= Z_1^{\beta_1} \cdots Z_N^{\beta_N}$. Define
\begin{align*}
	H_{co}^m(\Omega)
		&= \set{f \in L^2(\Omega) \colon Z^\beta f \in L^2(\Omega) \text{ for all } \abs{\beta} \le m}
\end{align*}
with
\begin{align*}
	\norm{f}_{H_{co}^m(\Omega)}^2
		= \sum_{\abs{\beta} \le m} \norm{Z^\beta f}^2_{L^2(\Omega)}.
\end{align*}
We say that $f$ is in the space, $W^{m, \infty}_{co}$, if
\begin{equation*} 
\| f \|_{W^{m, \infty}_{co}} := \sum_{|\beta| \leq m} \| Z^\beta u \|_{L^\iny(\Omega)} < \infty
\end{equation*}
and we define the space $E^m$ by
\begin{equation*} 
E^m := \{ u \in H^{m}_{co}(\Omega) | \text{ } \nabla u \in H^{m-1}_{co}(\Omega) \}
\end{equation*}
with the obvious norm.
\end{definition}

\Ignore{ 
\begin{theorem}[Masmoudi, Rousset \cite{MasmoudiRousset2010}]\label{T:MR}
Let $m$ be an integer satisfying $m>6$ and $\Omega$ be a $C^{m+2}$ domain. Consider $u_0 \in E^m\cap H$ such that $\nabla u_0 \in W^{1, \infty}_{co}$. Then there exists $T>0$ such that for every $0 < \ep <1$ and $\alpha$, $|\alpha| \leq 1$, there exists a unique $u^{\ep} \in \mathcal{C}([0,T], E^m)$ such that $\|  \nabla u^{\ep} \|_{1, \infty}$ is bounded on $[0,T]$ solution of \refEAnd{0-1}{General_NBC} with $f=0$. Moreover, there exists $C>0$ independent of $\ep$ and $\alpha$ such that
\begin{align}\label{e:MR1}
	\begin{split}
		\sup_{[0,T]} &\big( \|u^{\ep}(t)\|_{H_{co}^m(\Omega)}
		+ \|  \nabla u^{\ep}(t) \|_{H_{co}^{m - 1}(\Omega)}
		+ \| \nabla u^{\ep}(t) \|_{W^{1, \infty}_{co}} \big) \\
		&
		+ \ep \int_{0}^{T} \| \nabla^2 u(s) \|^2_{H_{co}^{m - 1}(\Omega)} \, ds
		\le C.
	\end{split}
\end{align}
\end{theorem}
} 
\begin{theorem}[Masmoudi, Rousset \cite{MasmoudiRousset2010}]\label{T:MR}
Let $m$ be an integer satisfying $m>6$ and $\Omega$ be a $C^{m+2}$ domain. Consider $u_0 \in E^m\cap H$ such that $\nabla u_0 \in W^{1, \infty}_{co}$. Then there exists $T>0$ such that for all sufficiently small $\eps$ there exists a unique solution, $u^{\ep} \in \mathcal{C}([0,T], E^m)$, to \refEAnd{0-1}{General_NBC} with $f=0$ and such that $\|  \nabla u^{\ep} \|_{1, \infty}$ bounded on $[0,T]$. Moreover, there exists $C = C(\ol{\al}) >0$, where $\ol{\al} = \norm{\Cal{A}}_{C^m(\Gamma)}$, such that
\begin{align}\label{e:MR1}
	\begin{split}
		\sup_{[0,T]} &\big( \|u^{\ep}(t)\|_{H_{co}^m(\Omega)}
		+ \|  \nabla u^{\ep}(t) \|_{H_{co}^{m - 1}(\Omega)}
		+ \| \nabla u^{\ep}(t) \|_{W^{1, \infty}_{co}} \big) \\
		&
		+ \ep \int_{0}^{T} \| \nabla^2 u(s) \|^2_{H_{co}^{m - 1}(\Omega)} \, ds
		\le C.
	\end{split}
\end{align}
\end{theorem}

\begin{remark}\label{R:SimpleExtension}
	\refTAndAnd{NSExistence}{EulerWellPosedness}{MR} were proved for a bounded domain, but each
	of the proofs extends easily to a 3D channel. \refTAnd{NSExistence}{MR} were also proved
	assuming that $\Cal{A} = \al I$,
	where $\al$ is a constant, but they easily extend to a general $\Cal{A}$ by using $\ol{\al} =
	\norm{\Cal{A}}_{C^{m}(\Gamma)}$ in place of $\al$ in certain boundary terms, much as we do
	in \refSAnd{ChannelCorrectorBounds}{CorrectorBoundsCurved}.
\end{remark}

%
%
\section{Channel domain}\label{S:S-2}

\noindent In this section, we prove \refE{68} for a periodic channel domain in $\mathbb{R}^3$. We set
$
	\Omega_{\infty} := \mathbb{R}^2 \times (0, h),
$
and consider solutions to \refEAnd{0-1}{General_NBC} in a channel domain $\Omega_{\infty}$. That is,
\begin{equation}\label{e:1}\left\{\begin{array}{l}
\spacer \dfrac{\pa u^{\ep}}{\pa t} -\ep \Delta u^{\ep} + (u^{\ep} \cdot \nabla)u^{\ep} + \nabla p^{\ep} = f, \text{ in } \Omega_{\infty}\times(0,T) ,\\
\spacer \text{div }u^{\ep} = 0, \text{ in } \Omega_{\infty} \times(0,T),\\
\spacer \text{$u^{\ep}$ and $p^{\ep}$ are periodic in $x$ and $y$ directions with periods $L_1$ and $L_2$},\\
u^{\ep}| _{t=0} = u_{0}, \text{ in } \Omega_{\infty}.
\end{array}
\right.
\end{equation}
Here, $f$ and $u_{0}$, satisfying \refE{0-6}, are assumed to be periodic in $x$ and $y$ directions with periods $L_1$ and $L_2$, respectively.

Since $\n = (0, 0, -1)$ at $z = 0$ and $\n = (0, 0, 1)$ at $z=h$, we can write the Generalized Navier boundary condition, appearing in \refE{General_NBC} with \refE{0-3}, in the form
\begin{equation}\label{e:5}\left\{\begin{array}{l}
            \spacer
            u^{\ep}_3 = 0, \text{ at } z=0, h,\\
            \displaystyle \spacer
            \dfrac{\pa u^{\ep}_i}{\pa z} - 2 \sum_{j=1}^2 \alpha_{ij} u^{\ep}_j =0, \text{ } i=1,2, \text{ at } z=0,\\
            \displaystyle
            \dfrac{\pa u^{\ep}_i}{\pa z} + 2\sum_{j=1}^2 \alpha_{ij} u^{\ep}_j =0, \text{ } i=1,2, \text{ at } z=h.
\end{array}\right.
\end{equation}

The corresponding limit problem, \refE{Euler}, can be written as
\begin{equation}\label{e:4}\left\{\begin{array}{l}
\spacer \dfrac{\pa u^0}{\pa t}  + (u^0 \cdot \nabla)u^0 + \nabla p^{0} = f, \text{ in } \Omega_{\infty}\times(0,T) ,\\
\spacer \text{div }u^0 = 0, \text{ in } \Omega_{\infty}\times(0,T),\\
\spacer \text{$u^0$ and $p^{0}$ are periodic in $x$ and $y$ directions with periods $L_1$ and $L_2$},\\
\spacer u^0_3 = 0, \text{ at } z= 0, h,\\
u^0| _{t=0} = u_{0}, \text{ in } \Omega_{\infty}.
\end{array}
\right.
\end{equation}
For the sake of convenience, we set
\begin{align*}
\Omega := (0, L_1)\times(0, L_2) \times (0,h),
\end{align*}
and assume (to simplify the expressions in \refE{22}) that
\begin{align*}
	\ep < (h/8)^2.
\end{align*}

To study the boundary layer associated with the Navier-Stokes problem \refE{1} with the Navier friction boundary conditions \refE{5}, we expand $u^{\ep}$ as
\begin{equation}\label{e:12}
u^{\ep} \simeq u^0 + \tht^{\ep} ,
\end{equation}
where $u^0$ is the solution of \refE{4} and $\tht^{\ep}$ is a divergence-free corrector, which will be determined below. The main role of $\tht^{\ep}$ is to correct the tangential error related to the normal derivative of $u^{\ep} - u^0$ on the boundary; see \refE{14} below.

\subsection{The corrector}\label{S:S2.1}
To define a corrector, $\tht^{\ep} = (\tht^{\ep}_1, \tht^{\ep}_2, \tht^{\ep}_3)$, using the ansatz $\tht^{\ep}_3 \simeq \ep^{1/2}\tht^{\ep}_i$, $i=1,2$, with respect to the order of $\ep$ in any Sobolev space, we first devote ourselves to find a suitable boundary condition for $\tht^{\ep}_i$, $i =1,2$. By inserting the expansion \refE{12} into \refE{5}$_{2,3}$, we find that, for $i=1,2$,
\begin{align*} 
\left\{\begin{array}{l}
            \spacer \displaystyle
            \dfrac{\pa u^0_i}{\pa z} - 2 \sum_{j=1}^2 \alpha_{ij} u^0_j  + \dfrac{\pa \tht^{\ep}_i}{\pa z} - 2 \sum_{j=1}^2 \alpha_{ij} \tht^{\ep}_j  \simeq 0,  \text{ at } z=0,\\
            \displaystyle
            \dfrac{\pa u^0_i}{\pa z} + 2 \sum_{j=1}^2 \alpha_{ij} u^0_j  + \dfrac{\pa \tht^{\ep}_i}{\pa z} + 2 \sum_{j=1}^2 \alpha_{ij} \tht^{\ep}_j  \simeq 0, \text{ at } z=h.
\end{array}\right.
\end{align*}
For smooth $\alpha_{ij}$, $1 \leq i,j \leq 2$ on $\Gamma$, independent of $\ep$, we expect that $\pa \tht^{\ep}_i/\pa z \gg 2 \sum_{j=1}^2 \alpha_{ij} \tht^{\ep}_j $, $i=1,2$. Hence, we use the Neumann boundary condition for $\tht^{\ep}_i$,
\begin{equation}\label{e:14}
\left\{\begin{array}{l}
            \spacer \displaystyle
            \dfrac{\pa \tht^{\ep}_i}{\pa z} = \widetilde{u}_{i, L}:= - \Big(\dfrac{\pa u^0_i}{\pa z} - 2  \sum_{j=1}^2 \alpha_{ij} u^0_j \Big), \text{ at } z=0 \text{ for } i = 1, 2,\\
            \displaystyle
            \dfrac{\pa \tht^{\ep}_i}{\pa z} = \widetilde{u}_{i, R} := - \Big( \dfrac{\pa u^0_i}{\pa z} + 2  \sum_{j=1}^2 \alpha_{ij} u^0_j\Big), \text{ at } z=h  \text{ for } i = 1, 2.
\end{array}\right.
\end{equation}

In the theory of boundary layer analysis, it is well known that the Neumann type boundary condition, \refE{14}, is useful when treating any weak boundary layer phenomenon. More precisely, to improve the convergence given in \refE{c9}, it is sufficient to construct a corrector function that fixes the normal derivative of the difference $u^{\ep} - u^0$ on the boundary, instead of the difference itself.

Toward this end, we  first define cutoff functions, $\sig_L, \sig_R$, belonging to $\mathcal{C}^{\infty}([0,h])$, by\begin{equation}\label{e:22}
\sig_{L}(z) := \left\{\begin{array}{ll}
                            \spacer 1, & 0 \leq z \leq h/8,\\
                            0, & h/4 \leq z \leq h. \end{array}\right.,
	\quad
	\sig_R (z) : = \sig_L(h-z).
\end{equation}
Then
we define the tangential component $\tht^{\ep}_i$, $i=1,2$, of the corrector $\tht^{\ep} = (\tht^{\ep}_1, \tht^{\ep}_2, \tht^{\ep}_3)$ as
\begin{equation}\label{e:23-1}
\tht^{\ep}_i := \tht^{\ep}_{i, L} + \tht^{\ep}_{ i, R}, \text{ }i=1,2,
\end{equation}
where
\begin{align}\label{e:24}
	\begin{split}
		\tht^{\ep}_{i, L}
			&= - \sqrt{\ep} \widetilde{u}_{i, L}(x,y;t) \sig_L(z) e^{-\frac{z}{\sqrt{\ep}}}
				- \ep \widetilde{u}_{i, L}(x,y;t) \sig^{\pri}_L(z)
					\Big(1 - e^{-\frac{z}{\sqrt{\ep}}}\Big)  \\
			&= - \ep \widetilde{u}_{i, L}(x,y;t) \dfrac{\pa}{\pa z} \Big\{ \sig_L(z)
				\Big(1 - e^{-\frac{z}{\sqrt{\ep}}}\Big) \Big\}, \\
			\tht^{\ep}_{i, R}
				&=  \sqrt{\ep} \widetilde{u}_{i, R}(x,y;t) \sig_R(z) e^{-\frac{h - z}{\sqrt{\ep}}}
					- \ep \widetilde{u}_{i, R}(x,y;t)
					\sig^{\pri}_R(z) \Big(1 - e^{-\frac{h - z}{\sqrt{\ep}}}\Big) \\
				&= - \ep \widetilde{u}_{i, R}(x,y;t) \frac{\pa}{\pa z} \Big\{ \sig_R(z)
					\Big(1 - e^{-\frac{h - z}{\sqrt{\ep}}}\Big) \Big\}.
	\end{split}
\end{align}
To make $\tht^{\ep}$ divergence-free, we must define the normal component $\tht^{\ep}_3$ of $\tht^{\ep}$ as
\begin{align*} 
\tht^{\ep}_3 = \tht^{\ep}_{3, L} + \tht^{\ep}_{3, R},
\end{align*}
 where
\begin{equation}\label{e:25}\begin{array}{l}
\spacer \displaystyle \tht^{\ep}_{3, L} = \ep  \Big( \dfrac{\pa \widetilde{u}_{1, L}}{\pa x}+ \dfrac{\pa \widetilde{u}_{2, L}}{\pa y}\Big)(x,y;t)  \sig_L(z) \Big(1 - e^{-\frac{z}{\sqrt{\ep}}}\Big),\\
\displaystyle \tht^{\ep}_{3, R} = \ep  \Big( \dfrac{\pa \widetilde{u}_{1, R}}{\pa x}+ \dfrac{\pa \widetilde{u}_{2, R}}{\pa y}\Big)(x,y;t)  \sig_R(z) \Big(1 - e^{-\frac{h - z}{\sqrt{\ep}}}\Big).
\end{array}
\end{equation}
(This form of the corrector is as in \cite{KTW2010}, adapted to Navier boundary conditions.)

Thanks to \refEAnd{22}{24}, by differentiating \refE{23-1} with respect to the normal variable $z$, one can easily verify that the tangential components, $\tht^{\ep}_1, \tht^{\ep}_2$, satisfy the desired boundary condition \refE{14}. Moreover, from \refE{25}, we infer that
\begin{equation}\label{e:27}
\tht^{\ep}_3 = 0, \text{ at } z=0,h.
\end{equation}
%

\subsection{Bounds on the corrector}\label{S:ChannelCorrectorBounds}
We introduce the following convenient notation:
\begin{align*}
	\dfrac{\pa^k}{\pa \tau^k} : =   \Big(\begin{array}{l}
                                                    \text{any differential operator of order $k$}\\
                                                    \text{with respect to the tangential variables $x$ and $y$}\end{array}\Big), \hspace{2mm} k \geq0.
\end{align*}
We also use the convention that $\kap_T = \kap_T(T, u_0, f)$ is a constant that depends on $T$, $u_0$, and $f$, but is independent of $\ep$ and $\mathcal{A}$, and may vary from occurrence to occurrence.

Letting
\begin{align}\label{e:olal}
	\ol{\al} = \norm{\Cal{A}}_{C^m(\Gamma)}, \, m>6,
\end{align}
\refEThrough{14}{25} give
\begin{equation}\label{e:bvofc1}
	\norm{\dfrac{\pa \tht^{\ep}_i}{\pa \tau}}_{L^{\infty}([0,T]\times\ov{\Omega})}
		\leq \kap_T(1+ \ov{\alpha})\ep^{\frac{1}{2}},
	\qquad
	\norm{\dfrac{\pa \tht^{\ep}_i}{\pa z}}_{L^{\infty}([0,T]\times\ov{\Omega})}
		\leq \kap_T(1+ \ov{\alpha}), \text{ } i=1,2,
\end{equation}
\begin{equation}\label{e:bvofc2}
	\norm{\dfrac{\pa \tht^{\ep}_3}{\pa \tau}}_{L^{\infty}([0,T]\times\ov{\Omega})}
		\leq \kap_T(1+ \ov{\alpha})\ep,
	\qquad
	\norm{\dfrac{\pa \tht^{\ep}_3}{\pa z}}_{L^{\infty}([0,T]\times\ov{\Omega})}
		\leq \kap_T(1+ \ov{\alpha})\ep^{\frac{1}{2}}.
\end{equation}

We have the following bounds on the corrector:

\begin{lemma}\label{L:2}
Assume \refE{0-6} holds and that $k, l, n \ge 0$ are integers either $l=1$, $k=0$ or $l=0$, $0 \leq k \leq 2$. Then the corrector, $\tht^{\ep}$, defined by \refEThrough{23-1}{25}, satisfies
\begin{equation}\label{e:31}\left\{\begin{array}{l}
\spacer \norm{ \dfrac{\pa^{l+k+n} \tht^{\ep}_i}{\pa t^l \pa \tau^k \pa z^n}}_{L^{\infty}(0,T;L^{2}(\Omega))} \leq C (1 + \ov{\alpha})\ep^{\frac{3}{4} - \frac{n}{2}}  , \text{ } i=1,2, \\
\spacer \norm{  \dfrac{\pa^{l+k} \tht^{\ep}_3}{\pa t^l \pa \tau^k }  }_{L^{\infty}(0,T;L^{2}(\Omega))} \leq C (1 + \ov{\alpha}) \ep,
\quad
\norm{ \dfrac{\pa^{l+k+n+1} \tht^{\ep}_3}{\pa t^l \pa \tau^k \pa z^{n+1}}}_{L^{\infty}(0,T;L^{2}(\Omega))} \leq C (1 + \ov{\alpha}) \ep^{\frac{3}{4}  - \frac{n}{2}}
\end{array}\right.
\end{equation}
for a constant, $C = C(T, l, k, n, u_0, f)$.
\end{lemma}
\begin{proof}
The assumptions \refE{0-6} give the regularity of $u^0$ in \refE{0-7}, and since $m \ge 5$, this allows $k$ to be at least as large as $2$.
To prove the lemma, using \refEThrough{22}{25}, we first notice that it is sufficient to verify \refE{31} with $\tht^{\ep}_i$ replaced by $\tht^{\ep}_{i,L}$, $1 \leq i \leq 3$.

For \refE{31}$_1$ with $\tht^{\ep}_{i,L}$, using \refE{24}$_1$, we write
\begin{equation}\label{e:pf0}
\dfrac{\pa^{l+k} \tht^{\ep}_{i, L}}{\pa t^l \pa \tau^k } = -\ep^{\frac{1}{2}} \dfrac{\pa^{l+k} \widetilde{u}_{i,L}}{\pa t^l \pa \tau^k } \big(\sig_L(z) - \ep^{\frac{1}{2}} \sig^{\pri}_{L}(z)\big) e^{-\frac{z}{\sqrt{\ep}}}  -\ep \dfrac{\pa^{l+k} \widetilde{u}_{i,L}}{\pa t^l \pa \tau^k }\sig^{\pri}_{L}(z).
\end{equation}
Then, by differentiating \refE{pf0} $n$ times in the $z$ variable, and using \refEAnd{14}{22}, we find
\begin{equation}\label{e:pf1}
\begin{array}{ll}
\abs{\dfrac{\pa^{l+k+n} \tht^{\ep}_{i, L}}{\pa t^l \pa \tau^k \pa z^n}}
                                        & \hspace{-2mm} \displaystyle \leq C(1 + \ov{\alpha}) \ep^{\frac{1}{2}- \frac{n}{2}}  e^{-\frac{z}{\sqrt{\ep}}} + C (1 + \ov{\alpha}) \ep + e.s.t.,
\end{array}
\end{equation}
where $e.s.t.$ is a function (or a constant) whose norm in all Sobolev spaces $H^{s}$ (and thus spaces $\mathcal{C}^{s}$) is exponentially small with a bound of the form $c_{1}\exp(-c_{2}/\ep^{\gam})$, $c_{1}, c_{2}, \gam >0$, for each $s$. Hence, we find
\begin{equation}\label{e:pf2}\begin{array}{ll}
\norm{ \dfrac{\pa^{l+k+n} \tht^{\ep}_{i, L}}{\pa t^l \pa \tau^k \pa z^n}}_{L^{\infty}(0,T;L^{2}(\Omega))}
                                         & \hspace{-3mm} \displaystyle \leq C(1 + \ov{\alpha}) \ep^{\frac{1}{2}- \frac{n}{2}} \Big( \int_{0}^{h} e^{-\frac{2z}{\sqrt{\ep}}} \, dz \Big)^{\frac{1}{2}} + C (1 + \ov{\alpha}) \ep\\
                                         & \spacer \hspace{-3mm} \displaystyle \leq (\text{setting $z^{\pri} = z/\sqrt{\ep}$})\\
                                         & \spacer \hspace{-3mm} \displaystyle \leq C (1 + \ov{\alpha}) \ep^{\frac{3}{4}- \frac{n}{2}} \Big( \int_{0}^{\infty} e^{-2 z^{\pri}} \, dz^{\pri} \Big)^{\frac{1}{2}} + C (1 + \ov{\alpha}) \ep\\
                                         & \hspace{-3mm} \displaystyle \leq C (1 + \ov{\alpha}) \ep^{\frac{3}{4}- \frac{n}{2}} , \text{ for  } l,k,n \geq0.
\end{array}
\end{equation}

To prove \refE{31}$_2$ with $\tht^{\ep}_{3,L}$, using \refE{25}$_1$, we write
\begin{equation*}\label{e:pf3}
\dfrac{\pa^{l+k} \tht^{\ep}_{3, L}}{\pa t^l \pa \tau^k } = \ep \dfrac{\pa^{l+k} }{\pa t^l \pa \tau^k } \Big(  \dfrac{\pa \widetilde{u}_{1, L}}{\pa x}+ \dfrac{\pa \widetilde{u}_{2, L}}{\pa y} \Big)  \sig_L(z) \Big(1 - e^{- \frac{z}{\sqrt{\ep}}} \Big).
\end{equation*}
Hence, \refE{31}$_2$ with $\tht^{\ep}_{3,L}$ follows by applying exactly the same computations as \refEAnd{pf1}{pf2}, and  the proof of \refL{2} is complete.
\end{proof}

We define a continuous piecewise linear function, $\zeta(z)$, by
\begin{equation}\label{e:zeta}\begin{array}{ll}
\zeta(z) : = \left\{  \begin{array}{ll}
                                        z, & 0 \leq z \leq h/4,\\
                                        h/4, & h/4 \leq z \leq 3h/4,\\
                                        h-z, & 3h/4 \leq z \leq h.
\end{array}\right.
\end{array}
\end{equation}
Then, using the analog of the proof of Lemma \ref{L:2}, one can verify that, $i=1,2$,
\begin{equation}\label{e:special_est_Channel}
\norm{ \dfrac{\zeta(z)}{\sqrt{\ep}} \dfrac{\pa \tht^{\ep}_i}{\pa z}}_{L^{\infty}(0,T;L^{2}(\Omega))} \leq C (1 + \ov{\alpha})\ep^{\frac{1}{4}}  ,
\quad
\norm{ \dfrac{\zeta(z)}{\sqrt{\ep}} \dfrac{\pa \tht^{\ep}_3}{\pa z}}_{L^{\infty}(0,T;L^{2}(\Omega))} \leq C (1 + \ov{\alpha})\ep^{\frac{1}{2}}.
\end{equation}

\subsection{Error analysis}\label{S:S2.3} We set the remainder:
\begin{equation}\label{e:40}
w^{\ep} : = u^{\ep } - u^0 - \tht^{\ep}. 
\end{equation}
Then, using \refEThrough{1}{4} with \refEAndAnd{14}{27}{40}, the equations for $w^{\ep}$ read
\begin{equation}\label{e:41}\left\{\begin{array}{l}
            \spacer
            \dfrac{\pa w^{\ep}}{\pa t} -\ep \Delta w^{\ep} + \nabla \big(p^{\ep}-p^0\big)
                        = \ep \Delta u^0 + R_{\ep}(\tht^{\ep})  - J_{\ep}(u^\ep, u^0), \text{ in } \Omega_{\infty}\times(0,T) ,\\
            \spacer
            \text{div }w^{\ep} = 0, \text{ in } \Omega_{\infty} \times(0,T),\\
            \spacer
            \text{$w^{\ep}$ is periodic in $x$ and $y$ directions with periods $L_1$ and $L_2$},\\
            \spacer
             w^{\ep}_3 = 0, \text{ at } z=0, h,\\
             \spacer \displaystyle
             \dfrac{\pa w^{\ep}_i}{\pa z} - 2  \sum_{j=1}^2 \alpha_{ij} w^{\ep}_j
                        = 2 \sum_{j=1}^2 \alpha_{ij} \tht^{\ep}_j, \text{ } i=1,2, \text{ at } z=0,\\
             \spacer \displaystyle
             \dfrac{\pa w^{\ep}_i}{\pa z} + 2\sum_{j=1}^2 \alpha_{ij} w^{\ep}_j
                        = -2 \sum_{j=1}^2 \alpha_{ij} \tht^{\ep}_j, \text{ } i=1,2, \text{ at } z=h,\\
             w^{\ep}| _{t=0} = - \tht^{\ep}| _{t=0}, \text{ in } \Omega_{\infty},
\end{array}
\right.
\end{equation}
where
\begin{align}\
	R_{\ep}(v) &:= -\dfrac{\pa v}{\pa t} + \ep \Delta v, \text{ for any smooth vector field }v, \label{e:42} \\
	J_{\ep}(u^\ep, u^0) &:= (u^{\ep}\cdot\nabla)u^{\ep} - (u^{0}\cdot\nabla)u^{0}. \label{e:43}
\end{align}
We multiply \refE{41}$_1$ by $w^{\ep}$, integrate over $\Omega$ and then, integrate it by parts . As a result, after applying the Schwarz and Young inequalities as well, we find:
\begin{equation}\label{e:44}
	\begin{split}
		\dfrac{d}{d t} &\norm{w^{\ep}}_{L^2(\Omega)}^2 + 2 \ep \norm{\nabla w^{\ep}}_{L^2(\Omega)}^2
			\leq \ep^2\norm{\Delta u^0}_{L^2(\Omega)}^2
				+ \norm{R_{\ep}(\tht^{\ep})}_{L^2(\Omega)}^2 + 2 \norm{w^{\ep}}_{L^2(\Omega)}^2 \\
			& + 2 \ep \int_{\{z=0,h\}} \big( \nabla w^{\ep} \n \big)\cdot w^{\ep} \, dS
				- 2 \int_{\Omega} J_{\ep}(u^{\ep}, u^0) \cdot w^{\ep} \, d \blds{x}.
	\end{split}
\end{equation}
Thanks to \refL{2} and \refE{42} with $v$ replaced by $\tht^{\ep}$, we find that
\begin{equation}\label{e:49}
	\norm{R_{\ep}(\tht^{\ep})}_{L^2(\Omega)}^2
		\leq \kap_T (1+ \ov{\alpha}^2) \ep^{\frac{3}{2}}.
\end{equation}
On the other hand, by remembering that $\n = (0,0,-1)$ at $z=0$ and $\n = (0,0,1)$ at $z=h$ , we notice that
\begin{equation}\label{e:46}\begin{array}{ll}
\big[ \nabla w^{\ep} \n\big]_{\text{tan}}
                                                        & \spacer \hspace{-2mm} = \left\{\begin{array}{ll}
                                                                        \spacer - \Big(\dfrac{ \pa w^{\ep}_1}{\pa z} , \hspace{2mm}  \dfrac{ \pa w^{\ep}_2}{\pa z} \Big), & \text{ at } z=0,\\
                                                                        \Big(  \dfrac{\pa w^{\ep}_1}{\pa z} , \hspace{2mm} \dfrac{ \pa w^{\ep}_2}{\pa z} \Big), & \text{ at } z=h,\\
                                                                        \end{array}\right.\\
                                                        & \hspace{-2mm} = (\text{using \refE{41}$_{5,6}$})\\
                                                        & \hspace{-2mm} \displaystyle
                                                                    = -2 \Big( \sum_{j=1}^{2} \alpha_{1j} (w^{\ep}_j+ \tht^{\ep}_j), \;
                                                                            \sum_{j=1}^{2} \alpha_{2j} (w^{\ep}_j+ \tht^{\ep}_j) \Big), \text{ at } z=0,h.
\end{array}
\end{equation}
Then, using \refE{46}, we find that
\begin{align}\label{e:51}
	\begin{split}
		&2 \ep \Big| \int_{\{z=0,h\}} \big( \nabla w^{\ep} \n \big)\cdot w^{\ep} \, dS \Big|
			\le \kap_T \ov{\alpha} \ep \norm{ [w^{\ep} + \tht^{\ep}]_{\text{tan}}}_{L^2(\Gamma)}
				\norm{[w^{\ep}]_{\text{tan}}}_{L^2(\Gamma)} \\
			&\le \kap_T \ov{\alpha} \ep \norm{w^{\ep}}_{L^2(\Gamma)}^2
				+ \kap_T \ov{\alpha} \ep \norm{[\tht^{\ep}]_{\text{tan}}}_{L^2(\Gamma)}
				\norm{w^{\ep}}_{L^2(\Gamma)} \\
			&\le (\text{using \refL{trace-like}, \refE{24}, and the Poincar\'{e} inequality}) \\
			&\le \kap_T \ov{\alpha} \ep \norm{w^{\ep}}_{L^2(\Omega)}
				\norm{\nabla w^{\ep}}_{L^2(\Omega)} + \kap_T (1 + \ov{\alpha}^2) \ep^{\frac{3}{2}}
				\norm{\nabla w^{\ep}}_{L^2(\Omega)} \\
			&\le \ep \norm{\nabla w^{\ep}}_{L^2(\Omega)}^2
				+ \kap_T \ov{\alpha}^2  \ep \norm{w^{\ep}}_{L^2(\Omega)}^2
				+ \kap_T ( 1 + \ov{\alpha}^4 ) \ep^2.
	\end{split}
\end{align}
By applying \refEAnd{49}{51} to \refE{44}, we obtain
\begin{align}\label{e:52}
	\begin{split}
	\dfrac{d}{d t} &\norm{w^{\ep}}_{L^2(\Omega)}^2
			+  \ep \norm{\nabla w^{\ep}}_{L^2(\Omega)}^2 \\
                  &\leq  \kap_T (1 + \ov{\alpha}^4) \ep^{\frac{3}{2}}
                    	+ \kap_T(1 + \ov{\alpha}^2 ) \norm{w^{\ep}}_{L^2(\Omega)}^2
			- 2 \int_{\Omega} J_{\ep}(u^{\ep}, u^0) \cdot w^{\ep} \, d \blds{x}.
	\end{split}
\end{align}
To estimate the last term on the right-hand side of \refE{52}, using \refEAnd{40}{43}, we first notice that
\begin{equation}\label{e:53}
J_{\ep}(u^{\ep}, u^0) =
             (u^{\ep}\cdot \nabla)w^{\ep} + (w^{\ep} \cdot \nabla)(u^{\ep} - w^{\ep}) + (u^{0}\cdot \nabla)\tht^{\ep} + (\tht^{\ep}\cdot \nabla)u^{0} + (\tht^{\ep}\cdot \nabla)\tht^{\ep}.
\end{equation}
Then, we write:
\begin{equation}\label{e:54}
 \int_{\Omega} J_{\ep}(u^{\ep}, u^0) \cdot w^{\ep} \, d \blds{x} := \sum_{j=1}^{5} \mathcal{J}_{\ep}^j ,
\end{equation}
where
\begin{align}\label{e:55}
	\left\{
	\begin{array}{ll}
		\displaystyle
		\mathcal{J}_{\ep}^1
			= \int_{\Omega} (u^{\ep}\cdot \nabla)w^{\ep}  \cdot w^{\ep} \, d \blds{x} = 0,
		&
		\displaystyle
		\mathcal{J}_{\ep}^2
			= \int_{\Omega} (w^{\ep} \cdot \nabla)(u^{\ep} - w^{\ep}) \cdot w^{\ep} \, d \blds{x}, \\
		\\
		\displaystyle
		\mathcal{J}_{\ep}^3
			= \int_{\Omega} (\tht^{\ep} \cdot \nabla) u^0  \cdot w^{\ep} \, d \blds{x},
		&
		\displaystyle
		\mathcal{J}_{\ep}^4
			= \int_{\Omega} (u^{0}\cdot \nabla)\tht^{\ep}  \cdot w^{\ep} \, d \blds{x}, \\
		\\
		\displaystyle
		\mathcal{J}_{\ep}^5
			= \int_{\Omega}   (\tht^{\ep}  \cdot \nabla )\tht^{\ep}   \cdot w^{\ep} \, d \blds{x}.
	\end{array}
	\right.
\end{align}
To bound $\mathcal{J}_{\ep}^2$, using \refE{40}, we first write
\begin{align*}
\mathcal{J}_{\ep}^2 = \int_{\Omega} (w^{\ep} \cdot \nabla) u^0  \cdot w^{\ep} \, d \blds{x} +
\int_{\Omega} (w^{\ep} \cdot \nabla)\tht^{\ep} \cdot w^{\ep} \, d \blds{x}.
\end{align*}
Then
\begin{align*}
\Big|\int_{\Omega} (w^{\ep} \cdot \nabla) u^0  \cdot w^{\ep} \, d \blds{x} \Big|
                    \leq \kap_T \norm{\nabla u^0}_{L^{\infty}(\Omega)}
                    \norm{w^{\ep}}_{L^2(\Omega)}^2
                    \leq \kap_T \norm{w^{\ep}}_{L^2(\Omega)}^2
\end{align*}
and, thanks to \refEAnd{bvofc1}{bvofc2},
\begin{align*}
	\Big|\int_{\Omega} (w^{\ep} \cdot \nabla)\tht^{\ep} \cdot w^{\ep} \, d \blds{x} \Big|
		\leq \kap_T \norm{\nabla \tht^{\ep}}_{L^{\infty}(\Omega)}
			\norm{w^{\ep}}_{L^2(\Omega)}^2
			\leq \kap_T (1 + \ov{\alpha})  \norm{w^{\ep}}_{L^2(\Omega)}^2.
\end{align*}
Thus,
\begin{equation}\label{e:60}
	\norm{\mathcal{J}_{\ep}^2}
		\leq \kap_T (1 + \ov{\alpha}) \norm{w^{\ep}}_{L^2(\Omega)}^2.
\end{equation}
Using \refL{2}, we bound $\mathcal{J}_{\ep}^3$ by
\begin{equation}\label{e:61}
	\begin{split}
		\abs{\mathcal{J}_{\ep}^3}
                            & \leq \norm{\nabla u^0}_{L^{\infty}(\Omega)}
                            	\norm{\tht^{\ep}}_{L^2(\Omega)}
				\norm{w^{\ep}}_{L^2(\Omega)}
                            \leq \kap_T \norm{\tht^{\ep}}_{L^2(\Omega)}
                            	\norm{w^{\ep}}_{L^2(\Omega)} \\
                            & \leq \kap_T (1+ \ov{\alpha})\ep^{\frac{3}{4}}
                            		\norm{w^{\ep}}_{L^2(\Omega)}
                            \leq \kap_T (1 + \ov{\alpha}^2) \ep^{\frac{3}{2}}
                            + \norm{w^{\ep}}_{L^2(\Omega)}^2.
	\end{split}
\end{equation}
Since $u^0_3$ vanishes at $z=0 \text{ or } h$, using the regularity of $u^0$, we bound $\mathcal{J}_{\ep}^4$ by
\begin{equation}\label{e:62}
	\begin{split}
	\abs{\mathcal{J}_{\ep}^4}
			&\leq \sum_{j=1}^{3}
                            	\Big| \int_{\Omega} \Big( u^0_1\dfrac{\pa \tht^{\ep}_j}{\pa x}
					+ u^0_2\dfrac{\pa \tht^{\ep}_j}{\pa y} \Big) w^{\ep}_j \, d \blds{x}
				\Big|
				+ \sum_{j=1}^{3}
				\Big| \int_{\Omega}  u^0_3 \dfrac{\pa \tht^{\ep}_j}{\pa z} w^{\ep}_j \, d \blds{x}
				\Big|\\
			&\leq \norm{u^0}_{L^{\infty}(\Omega)} \sum_{j=1}^{3}
                            	\norm{\dfrac{\pa \tht^{\ep}_j}{\pa x}
					+ \dfrac{\pa \tht^{\ep}_j}{\pa y}}_{L^2(\Omega)}
				\norm{w^{\ep}_j}_{L^2(\Omega)} \\
			&\qquad\qquad
				+ \ep^{\frac{1}{2}} \sum_{j=1}^{3}
				\norm{\dfrac{u^0_3}{\zeta(z)}}_{L^{\infty}(\Omega)}
				\norm{\dfrac{\zeta(z)}{\sqrt{\ep}} \dfrac{\pa \tht^{\ep}_j}{\pa z}}_{L^2(\Omega)}
				\norm{w^{\ep}_j}_{L^2(\Omega)}  \\
			&\leq (\text{using \refEFirst{zeta}\refELast{special_est_Channel} and \refL{2}})\\
			&\leq \kap_T (1 + \ov{\alpha}) \ep^{\frac{3}{4}} \norm{w^{\ep}}_{L^2(\Omega)}
			\leq \kap_T (1 + \ov{\alpha}^2)\ep^{\frac{3}{2}}
                            	+ \norm{w^{\ep}}_{L^2(\Omega)}^2.
	\end{split}
\end{equation}
Thanks to \refEAnd{bvofc1}{bvofc2} and \refL{2}, was can bound $\mathcal{J}_{\ep}^5$ by
\begin{equation}\label{e:62-2}
	\begin{split}
	&\abs{\mathcal{J}_{\ep}^5}
		\leq \sum_{j=1}^{3} \Big| \int_{\Omega} \Big( \tht^{\ep}_1
			\dfrac{\pa \tht^{\ep}_j}{\pa x}
			+ \tht^{\ep}_2 \dfrac{\pa \tht^{\ep}_j}{\pa y} \Big)
			w^{\ep}_j \, d \blds{x} \Big|
			+ \sum_{j=1}^{3} \Big| \int_{\Omega} \tht^{\ep}_3
			\dfrac{\pa \tht^{\ep}_j}{\pa z} w^{\ep}_j \, d \blds{x} \Big| \\
		&\quad\leq \sum_{j=1}^{3}
			\Big\{ \norm{\tht^{\ep}_1}_{L^{\infty}(\Omega)}
				\norm{\dfrac{\pa \tht^{\ep}_j}{\pa x}}_{L^2(\Omega)}
				\norm{w^{\ep}_j}_{L^2(\Omega)}  +
				\norm{\tht^{\ep}_2}_{L^{\infty}(\Omega)}
				\norm{\dfrac{\pa \tht^{\ep}_j}{\pa y}}_{L^2(\Omega)}
				\norm{w^{\ep}_j}_{L^2(\Omega)} \Big\} \\
		&\qquad\qquad
                             + \sum_{j=1}^{3} \norm{\tht^{\ep}_3}_{L^{\infty}(\Omega)}
                             	\norm{ \dfrac{\pa \tht^{\ep}_j}{\pa z}}_{L^{2}(\Omega)}
				\norm{w^{\ep}_j}_{L^2(\Omega)}\\
		&\quad\leq \kap_T (1 + \ov{\alpha}^2) \ep^{\frac{5}{4}} \norm{w^{\ep}}_{L^2(\Omega)}
		\leq \kap_T (1 + \ov{\alpha}^4) \ep^{\frac{5}{2}} + \norm{w^{\ep}}^2_{L^2(\Omega)}.
	\end{split}
\end{equation}
Then, using \refEThrough{60}{62-2}, \refE{54} gives
\begin{equation}\label{e:65}
	\Big|  \int_{\Omega} J_{\ep}(u^{\ep}, u^0) \cdot w^{\ep} \, d \blds{x} \Big|
		\leq \kap_T  (1+\ov{\alpha}^4) \ep^{\frac{3}{2}}
			+ \kap_T (1 + \ov{\alpha} ) \norm{w^{\ep}}_{L^2(\Omega)}^2 .
\end{equation}
Applying \refE{65} to \refE{52}, we obtain
\begin{equation*}\label{e:66}
	\dfrac{d}{d t}\norm{w^{\ep}}_{L^2(\Omega)}^2
		+   \ep \norm{\nabla w^{\ep}}_{L^2(\Omega)}^2
		\leq \kap_T (1+ \ov{\alpha}^4)\ep^{\frac{3}{2}}
			+ \kap_T (1 + \ov{\alpha}^2) \norm{w^{\ep}}_{L^2(\Omega)}^2.
\end{equation*}
Moreover, using \refEThrough{23-1}{25} and \refE{41}$_7$, we observe that
\begin{equation*}\label{e:non_comp_channel}
\normmark  w^{\ep}|_{t=0} \normmark _{L^2(\Omega)} = \normmark  \tht^{\ep}|_{t=0} \normmark _{L^2(\Omega)} \leq \kap_T (1 + \ov{\alpha})\ep^{\frac{1}{2}}
\normmark e^{-\frac{z}{\sqrt{\ep}}} \normmark_{L^2(\Omega)} + l.o.t. \leq \kap_T (1 + \ov{\alpha}) \ep^{\frac{3}{4}}.
\end{equation*}
Thanks to the Gronwall inequality, we finally have the bounds on the remainder, $w^{\ep}$,
\begin{align}\label{e:67}
	\norm{w^{\ep}}_{L^{\infty}(0,T; L^2(\Omega))}
		\le \kap(T, \ov{\alpha}, u_0, f) \ep^{\frac{3}{4}},
		\quad
	\norm{w^{\ep}}_{L^{2}(0,T; H^1(\Omega))}
		\le \kap(T, \ov{\alpha}, u_0, f) \ep^{\frac{1}{4}}.
\end{align}

\subsection{Proof of convergence}\label{S:S2.4}
Using \refE{40}, we first notice that
\begin{equation}\label{e:69}
|u^{\ep} - u^0| \leq |w^{\ep}| + |\tht^{\ep} | \text{ } \text{ pointwise in $\Omega_{\infty} \times (0,T)$}.
\end{equation}
Then, using \refE{67} and \refL{2}, \refE{68} follows from \refE{69}. \hspace{30mm}$\Box$

%
%
\section{Bounded domain}\label{S:S-3}
In this section we consider the Navier-Stokes equations, \refEAnd{0-1}{General_NBC}, and the Euler equations, \refE{Euler}, in a bounded domain $\Omega$ in $\mathbb{R}^3$ with boundary $\Gamma$ having regularity as in \refE{0-6}. To handle the geometric difficulties of a curved boundary, we must treat $\Omega$ as a manifold with boundary, first constructing charts on $\Gamma = \prt \Omega$ in a special way, as we describe below.

We consider the boundary, $\Gamma$ as a submanifold of $\R^3$. Then, since $\Gamma$ is a compact and smooth surface in $\R^3$, we construct a system of finitely many charts where each chart is a $C^m$-map, $\widetilde{\psi}$, from a domain, $\widetilde{U}$, in $\R^{2}$ to a domain, $\widetilde{V}$, in $\Gamma$. More precisely, we choose an orthogonal curvilinear system $(\xi') = (\xi_1, \xi_2)$ in $\widetilde{U}$ so that, for any point $\widetilde{\blds{x}}$ on $\widetilde{V} \subset \Gamma$, we write
\begin{equation}\label{e:boundary_chart}
\widetilde{\blds{x}} =  \widetilde{\psi}(\xi'), \quad \xi' = (\xi_1, \xi_2) \in \widetilde{U}.
\end{equation}
Differentiating (\ref{e:boundary_chart}) with respect to $\xi_i$, $i=1,2$, variables, we obtain the covariant basis on $\widetilde{U}$ and the metric tensor:
\begin{equation}\label{e:boundary_co_basis}
\widetilde{\blds{g}}_i(\xi') := \dfrac{\pa \widetilde{\blds{x}}}{\pa \xi_i}, \quad i=1,2,
\end{equation}
and
\begin{equation}\label{e:boundary_tensor}
\big(  \widetilde{g}_{i j}(\xi')\big)_{1 \leq i,j \leq 2}
    := \big( \widetilde{\blds{g}}_i \cdot \widetilde{\blds{g}}_j \big)_{1 \leq i,j \leq 2}
    = \text{diag}\big(\widetilde{\blds{g}}_1 \cdot \widetilde{\blds{g}}_1, \, \,  \widetilde{\blds{g}}_2 \cdot \widetilde{\blds{g}}_2  \big).
\end{equation}
Moreover, we see that the determinant of the metric tensor is strictly positive;
\begin{equation}\label{e:positivity of boundary tensor}
\widetilde{g}(\xi') : = \text{det}(\widetilde{g}_{i j}) >0, \text{ for all } \xi' \text{ in the closure of } \widetilde{U}.
\end{equation}

For any smooth 2d compact manifold $\Gamma$ in $\R^3$, one can construct a system of finitely many charts, which satisfy \refEAnd{boundary_tensor}{positivity of boundary tensor}. Hence, the class of domains under consideration in this article covers all smooth bounded domains in $\R^3$ with boundary having regularity as in \refE{0-6}. Moreover, as we will see below in \refS{Intrinsic}, the construction of the corrector is independent of our choice of charts. Hence, it is sufficient to restrict our attention to a single chart only, since any estimates we develop will apply equally to all of $\Omega$.

We define $\Gamma_c$ to be the interior tubular neighborhood of $\Omega$ with width $c$ for any sufficiently small $c > 0$, and let $a > 0$ be small enough that $\Gamma_{3 a}$ is such a tubular neighborhood. We can globally define the coordinate $\xi_3$ on $\Gamma_{3 a}$ to be distance from the boundary, with positive distances directed inward.

We fix the orientation of $\xi'$ variables on $\widetilde{V}$ so that
\begin{equation}\label{e:normal_vector}
\blds{n}(\xi')  : = -\dfrac{\widetilde{\blds{g}}_1 \times \widetilde{\blds{g}}_2}{|\widetilde{\blds{g}}_1 \times \widetilde{\blds{g}}_2|}(\widetilde{\psi}(\xi')),
\end{equation}
where $\blds{n}(\xi')$ is the unit outer normal vector on $\widetilde{V}$. Then, letting $U = \widetilde{U} \times (0, 3 a)$,  we define a chart $\psi \colon U \to V$ (giving what are sometimes called \textit{boundary normal coordinates}):
\begin{align}\label{e:x}
	\blds{x}
		= \psi(\blds{\xi})
		= \widetilde{\psi}(\xi') - \xi_3 \n(\xi'),
        \quad \blds{x} = (x_1, x_2, x_3) \in \Gamma_{3a}.
\end{align}
By differentiating $\psi$ in $\blds{\xi}$ variables and using \refE{boundary_co_basis}, we define the covariant basis of the curvilinear system $\blds{\xi}$:
\begin{align}\label{e:basis}
	\blds{g}_i (\blds{\xi})
		&= \widetilde{\blds{g}}_i(\xi') - \xi_3 \pdx{\blds{n}}{\xi_i} (\xi'), \; i = 1, 2,
	\quad
	\blds{g}_3 (\blds{\xi})
		= - \n(\xi');
\end{align}
hence, from \refEAnd{normal_vector}{basis}, we see that the covariant basis satisfies the right-hand rule.\\

One important observation here is that the orthogonality, on $\widetilde{V}$, of $\widetilde{\g}_i$, $i=1,2$, does not imply the orthogonality, in $V$, of $\g_i$, $i=1,2,3$. To see this, we first notice that
\begin{equation}\label{e:derivatives of normal vector}
\dfrac{\pa \n}{\pa \xi_i} = (\text{linear combination of $\widetilde{\g}_1$ and $\widetilde{\g}_2$}), \quad i=1,2.
\end{equation}
Thus, $\g_i \cdot \g_3 = 0$ for $i=1,2$, but $\g_1 \cdot \g_2 \neq 0$ in general. Consequently, the metric tensor $( g_{i j}(\blds{\xi}))_{1 \leq i,j \leq 3} := ( \blds{g}_i \cdot \blds{g}_j \big)_{1 \leq i,j \leq 3}$ satisfies:
\begin{equation}\label{e:206}
    \left\{
    \begin{array}{l}
        \spacer
        g_{i j} = 0 \quad i=3 \text { and } j=1,2, \text{ or } i=1,2 \text{ and } j=3,\\
        g_{3 3} = 1.
    \end{array}
    \right.
\end{equation}
Moreover, thanks to (\ref{e:positivity of boundary tensor}), by choosing the thickness $3a>0$ of the tubular neighborhood $\Gamma_{3a}$  small enough, we see that
\begin{equation}\label{e:positivity of tensor}
g(\blds{\xi}) := \text{det}(g_{i j})_{1 \leq i,j \leq 3} > 0  \text{ for all } \blds{\xi} \text{ in the closure of } U=\widetilde{U}\times(0, 3a);
\end{equation}
The function, $\sqrt{g} := g^{1/2}$, is the magnitude of the Jacobian determinant of  the chart, $\psi$.

The matrix of the contravariant metric components are defined in the closure of $U$ as well:
\begin{equation}\label{e:cont_tensor}
    (g^{ij})_{1 \leq i,j \leq 3}
        = (g_{ij})^{-1}_{1 \leq i,j \leq 3}
        = \dfrac{1}{g}
        \left(
                \begin{array}{ccc}
                        g_{22}  &   -g_{12} &   0\\
                        -g_{12} &   g_{11}  &   0\\
                        0   & 0  & 1
                \end{array}
        \right).
\end{equation}

We introduce the normalized covariant vectors:
\begin{equation}\label{e:207}\begin{array}{ll}
\boldsymbol{e}_{i} = \dfrac{\boldsymbol{g_{i}}}{|\boldsymbol{g}_{i}|}, \text{ } 1 \leq i \leq 3.
\end{array}
\end{equation}
Then, for a vector valued function $F$, defined on $U$, in the form
\begin{equation*}\label{e:210}
F = \sum_{i=1}^{3} F_{i}\boldsymbol{e}_{i} ,
\end{equation*}
one can classically express the divergence
operator acting on $F$ in the $\boldsymbol{\xi}$ variable (see \cite{Cia05} or \cite{Klingenberg1978}) as
\begin{align}\label{e:211}
	\dv  F
		= \frac{1}{\sqrt{g}} \sum_{i=1}^2 \pdx{}{\xi_i} \pr{\frac{\sqrt{g}}{\sqrt{g_{ii}}} F_{i}}
			+ \frac{1}{\sqrt{g}} \pdx{(\sqrt{g} F_3)}{\xi_3}.
\end{align}

We write the Laplacian of $F$ as
\begin{equation}\label{e:214}
\Delta F = \sum_{i=1}^{3} \Big(\mathcal{S}^{i}F + \mathcal{L}^{i}F_{i} + \dfrac{\pa^{2}F_{i}}{\pa \xi_{3}^{2}} \Big) \blds{e}_i,
\end{equation}
where
\begin{equation}\label{e:215}\left\{\begin{array}{l}
\spacer
\mathcal{S}^{i} F= \Big( \begin{array}{l}
                                                    \text{linear combination of tangential derivatives}\\
                                                    \text{of $F_j$, $1\leq j \leq 3$, in $\xi'$, up to order 2}
                                                    \end{array} \Big),\\
\mathcal{L}^{i} F_{i} = \Big(\text{proportional to } \dfrac{\pa F_i}{\pa \xi_3}\Big).
\end{array}\right.
\end{equation}

\begin{rem}\label{R:201}
Note that the coefficients of $\mathcal{S}^{i}$ and $\mathcal{L}^{i}$, $1\leq i \leq3$ in \refE{215}, are
multiples of $\sqrt{g}$ , $1/\sqrt{g}$, $\sqrt{g_{ii}}$ , $1/\sqrt{g_{ii}}$, $i=1,2$, $g_{12}$, $g_{21}$, and their derivatives. Thanks to \refE{positivity of tensor}, all these quantities are  well-defined because of the regularity assumed for $\Gamma$ in \refE{0-6}.
\end{rem}

\begin{rem}\label{R:No mixing term}
Thanks to \refE{206}, we notice that the tangential directions are perpendicular to the normal direction in the tubular neighborhood $\Gamma_{3a}$. Indeed this property enables us to obtain the expression of Laplacian as \refEAnd{214}{215}, which is essentially the same as for the case of orthogonal curvilinear system. The explicit expression of Laplacian in orthogonal system appears in, e.g., \cite{GHT2}.
\end{rem}

For smooth vector fields $F, \text{ } G : U \rightarrow \mathbb{R}^3$, we consider $\nabla_{F} G$, the covariant derivative of $G$ in the direction $F$, which gives $F \cdot \nabla G$ in the Cartesian coordinate system. More precisely, we consider the smooth functions $F$ and $G$ in the form
\begin{equation*}\label{e:217}\begin{array}{lll}
\displaystyle F = \sum_{i=1}^{3} F_{i} \boldsymbol{e}_i, && \displaystyle G = \sum_{i=1}^{3} G_{i} \boldsymbol{e}_i.
\end{array}
\end{equation*}
Then, one can write $\nabla_{F} G$ in the $\boldsymbol{\xi}$ variable,\\

\begin{equation}\label{e:218}
\nabla_{F} G  = \sum_{i=1}^{3} \Big\{\mathcal{P}_i \Big(F_1, F_2 : \dfrac{\pa G_i}{\pa \xi_1}, \dfrac{\pa G_i}{\pa \xi_2}\Big)
                                + F_3 \dfrac{\pa G_i}{\pa \xi_3}
                                + \mathcal{Q}_i (F : G) \Big\}\blds{e}_i,
\end{equation}
where
\begin{equation}\label{e:219}
\left\{\begin{array}{l}
\spacer
\mathcal{P}_i \Big(F_1, F_2 : \dfrac{\pa G_i}{\pa \xi_1}, \dfrac{\pa G_i}{\pa \xi_2}\Big)
            = \left(
                        \begin{array}{l}
                        \text{product of a linear combination of}\\
                        \text{the tangential components $F_1$, $F_2$ of $F$,}\\
                        \text{and sum of tangential derivatives of $G_i$}
                        \end{array}
                \right),\\
\mathcal{Q}_i (F : G)
            = (\text{linear combination of the products $F_jG_k$, $1 \leq j,k \leq 3$ }).
\end{array}\right.
\end{equation}
\begin{rem}\label{R:Convective}
$\mathcal{Q}_i (F : G)$, $1 \leq i \leq 3$, are related to the Christoffel symbols of the second kind, which comes from the twisting effects of the curvilinear system $\boldsymbol{\xi}$. For the case of an orthogonal system, the explicit expression of \refE{218} is given in Appendix 2 of \cite{Bat99}.
\end{rem}

Using the expression of contravariant components of the strain rate tensor, and by remembering, from \refEFirst{boundary_tensor}\refENext{basis}\refELast{206}, that the covariant basis (and hence the normalized covariant basis) is triply orthogonal on $\Gamma$, we write the generalized Navier boundary conditions, \refE{General_NBC}, for $F = \sum_{i=1}^{3} F_i \blds{e}_i$ as
\begin{equation}\label{e:219-9}\left\{\begin{array}{l}
        \spacer
        F_3 = 0, \text{ at } \xi_3 =0,\\
        \displaystyle
        - \dfrac{1}{2} \dfrac{\pa F_i}{\pa \xi_3} + \mathcal{M}_i\Big(F_i, \dfrac{\pa F_3}{\pa \xi_i}\Big)
        + \sum_{j=1}^{2}\alpha_{ij} F_j = 0, \text{ in } \blds{e}_i \text{ direction at } \xi_3 =0,  \text{ } i=1,2,
\end{array}\right.
\end{equation}
where
\begin{equation}\label{e:219-10}
\mathcal{M}_i\Big(F_i, \dfrac{\pa F_3}{\pa \xi_i} \Big)
= \Big(
                                \begin{array}{l}
                                        \text{linear combination of the tangential component $F_i$ and}\\
                                        \text{the derivative in $\xi_i$ of the normal component $F_3$}
                                \end{array}\Big).
\end{equation}

\begin{rem}\label{R:NavierBC}
Thanks to \refEAnd{0-6}{positivity of tensor}, the coefficients of $\mathcal{M}_i(F)$, $i=1,2$, are well-defined. Concerning an orthogonal system, the explicit expression of $\mathcal{M}_i(F)$, $i=1,2$, appears on p. 115 of \cite{MF53}.
\end{rem}

\subsection{The corrector}\label{S:S3.1}
In defining the corrector, $\theta^\eps$, we parallel the strategy we used in \refS{S-2} for a periodic channel domain as closely as possible,
employing an asymptotic expansion, $u^{\ep} \simeq u^0 + \tht^{\ep}$, as in \refE{12}, but adapting the corrector to the curved boundary.

With the unit vectors, $\blds{e}_i$, defined as in \refE{207}, on $U$, $\theta^\eps$ can be written
\begin{equation}\label{e:221}
	\tht^\eps
		:= \sum_{i=1}^{3} \tht^\eps_i \blds{e}_i.
\end{equation}
The tangential components $\tht^{\ep}_{i}$, $i=1,2$, will be constructed to correct the tangential discrepancy in the boundary conditions related to the normal derivative of $u^{\ep} - u^0$ on the boundary. Then the normal component $\tht^{\ep}_3$ will be deduced from the divergence-free condition on $\tht^{\ep}$.

To define the corrector $\tht^{\ep}$ appearing in \refE{221}, since $u^0 \cdot \n =0$ on $\Gamma$, we first set
\begin{align*}
	\widetilde{u}
		= 2 \pr{[\mathbf{S}(u^0)  \n ]_{\text{tan}} + \mathcal{A} u^0} ,
\end{align*}
\Ignore{ 
\begin{align*}
	\widetilde{u} (\xi' ; t)
		= 2 \Big(   [\mathbf{S}(u^0)  \n ]_{\text{tan}}\big|_{\widetilde{V}} + \mathcal{A} u^0 |_{\widetilde{V}} \Big) ,
\end{align*}
} 
defined on all of $\Gamma$, and write, in coordinates,
\begin{equation}\label{e:224}
    \widetilde{u} (\xi'; t)  = \sum_{i=1}^2 \widetilde{u}_i (\xi'; t)  \blds{e}_i|_{\xi_3 = 0},
    \quad \quad
	\widetilde{u}_i(\xi' ; t)
		:= \widetilde{u} (\xi'; t) \cdot \blds{e}_i|_{\xi_3=0}.
\end{equation}
Then, we insert the expansion $u^{\ep} \simeq u^0 + \tht^{\ep}$ into the generalized Navier boundary conditions, \refE{General_NBC}, and, thanks to \refE{219-9}$_2$, we find that, for $i=1,2$,
\begin{equation*}\label{e:222}
\dfrac{1}{2}\widetilde{u}_i(\xi' ; t) -  \dfrac{1}{2} \dfrac{\pa \tht^{\ep}_i}{\pa \xi_3} + \mathcal{M}_i\Big(\tht^{\ep}_i, \dfrac{\pa \tht^{\ep}_3}{\pa \xi_i} \Big) + \sum_{j=1}^2 \alpha_{ij} \tht^{\ep}_j  \simeq 0, \text{ at } \xi_3 = 0.
\end{equation*}
Using \refE{219-10}, we expect that $\pa \tht^{\ep}_i/\pa\xi_3 \gg \mathcal{M}_i(\tht^{\ep}_i, \pa\tht^{\ep}_3/\pa\xi_i)$ or $\sum_{j=1}^2 \alpha_{ij} \tht^{\ep}_j$, $i=1,2$, for smooth $\alpha_{ij}$, $1 \leq i,j \leq 2$, independent of $\ep$.
Hence, we use the Neumann boundary condition for $\tht^{\ep}_i$,
\begin{equation}\label{e:223}
	\dfrac{\pa \tht^{\ep}_i }{\pa \xi_3} \Big|_{\xi_3 =0 }
		= \widetilde{u}_i (\xi';t), \text{ } i=1,2.
\end{equation}

We can now model the corrector after the flat-space version in \refE{24}. We  define a smooth cutoff function, $\sig(\xi_3)$, with
\begin{equation}\label{e:225}
\sig(\xi_3) := \left\{\begin{array}{ll}
                            \spacer 1, & 0 \leq \xi_3 \leq a,\\
                            0, &  \xi_3 \geq 2a. \end{array}\right.
\end{equation}
Letting
\begin{align*}
	\gamma_i
		:= \gamma_i(\xi') = \frac{\sqrt{g}}{\sqrt{g_{ii}}} \Big|_{\xi_3 =0},
\end{align*}
we define the tangential components of the corrector by
\begin{align}\label{e:226}
	\theta^\eps_i (\blds{\xi} ; t)
		& := - \eps \frac{\sqrt{g_{ii}}(\blds{\xi})}{\sqrt{g}(\blds{\xi})}
			\brac{
				(\gamma_i \widetilde{u}_i)(\xi' ; t) \pdx{}{\xi_3}
				\pr{\sigma(\xi_3)
				\pr{1 - e^{-\frac{\xi_3}{\sqrt{\eps}}}}}
			},
			i = 1, 2.
\end{align}
It follows from \refE{211} that
\begin{align*}
	\sqrt{g} \dv \theta^\eps
		&= - \eps
			\sum_{i=1}^2 \pdx{}{\xi_i} (\gamma_i \widetilde{u}_i)(\xi' ; t) \pdx{}{\xi_3}
				\pr{\sigma(\xi_3)
				\pr{1 - e^{-\frac{\xi_3}{\sqrt{\eps}}}}}
			+ \pdx{(\sqrt{g} \theta^\eps_3)}{\xi_3} \\
		&= - \eps
			\pdx{}{\xi_3} \Big\{\sum_{i=1}^2 \pdx{}{\xi_i} (\gamma_i \widetilde{u}_i)(\xi' ; t) \sigma(\xi_3)
				\pr{1 - e^{-\frac{\xi_3}{\sqrt{\eps}}}} \Big\}
			+ \pdx{(\sqrt{g} \theta^\eps_3)}{\xi_3}.
\end{align*}
Thus, we can easily ensure that $\theta^\eps$ is divergence-free by letting
\begin{align}\label{e:227}
	\theta^\eps_3 (\blds{\xi} ; t)
		& := \ep \frac{1}{\sqrt{g}(\blds{\xi})} \Big\{\sum_{i=1}^2 \pdx{}{\xi_i} (\gamma_i \widetilde{u}_i)(\xi' ; t) \Big\}
                 \sigma(\xi_3) \pr{1 - e^{-\frac{\xi_3}{\sqrt{\eps}}}}.
\end{align}

It is easy to see that $\theta^\eps_3$ vanishes at $\xi_3 = 0$, and by differentiating \refE{226} and using \refE{225}, we see that each tangential component $\tht^{\ep}_i$, $i=1,2$, satisfies the boundary condition \refE{223} to within order $\sqrt{\eps}$:
\begin{align}\label{e:223Approx}
	\dfrac{\pa \tht^{\ep}_i}{\pa \xi_3}\Big|_{\xi_3 = 0}
		= \widetilde{u}_i (\xi' ; t)
			- \sqrt{\eps} E(\xi'; t),
\end{align}
where
\begin{align}\label{e:Error_Boundary}
	E(\xi'; t)
		= \gamma_i(\xi')
				\pdx{}{\xi_3} \pr{\frac{\sqrt{g_{ii}}}{\sqrt{g}}}\Big|_{\xi_3 = 0}
                \widetilde{u}_i(\xi' ; t).
\end{align}

Due to the presence of $\sigma$ in \refEAnd{226}{227}, we also have
\begin{equation}\label{e:229}
\dfrac{\pa^k \tht^{\ep}_{i}}{\pa\xi_{j}^k}\Big|_{\xi_3 \geq 2a} = 0, \text{ } 1\leq i, j \leq3, \text{ }k \geq0.
\end{equation}

\begin{remark}\label{R:FlatBoundaryCorrector}
For the case of a flat boundary $\Gamma$, one can choose curvilinear coordinates with the metric tensor $(g_{ij})_{1 \leq i,j \leq 3}$, defined in \refE{206}, as the identity matrix $\blds{I}_{3 \times 3}$, and hence $\sqrt{g}$, $\sqrt{g_{ii}}$ and $\gamma_i$, $i=1,2$, appearing in \refEFirst{226}\refENext{227}\refELast{Error_Boundary}, are equal to $1$. This implies that the expression of the corrector defined by \refEAnd{226}{227} are identical to \refEAnd{24}{25} in a channel domain where the error $E$ in \refE{Error_Boundary} is now equal to $0$.
\end{remark}

\begin{remark}
The form of our corrector \refEAnd{226}{227} is similar to the background flow in Lemma 1 of \cite{W1997A}.
\end{remark}

%
%
\subsection{The corrector in principal curvature coordinates}\label{S:Intrinsic}
In this section, we express the corrector in particularly convenient and geometrically meaningful coordinates called \textit{principal curvature coordinates}. 

We define an umbilical point of $\Gamma$ to be a point at which the principal curvatures, $\kap_1$ and $\kap_2$, are equal (this also includes what some authors refer to as a planar point, where both curvatures vanish). By Lemma 3.6.6 of \cite{Klingenberg1978}, in some neighborhood of any non-umbilical point there exists a chart in which the metric tensor of \refE{boundary_co_basis} is diagonal (as is the second fundamental form) and the coordinate lines are parallel to the principal directions at each point. Such a chart is also called a \textit{principal curvature coordinate system}.

For now, we assume that we are working in such a chart, $\widetilde{\psi}_p \colon \widetilde{U}_p \to \widetilde{V} \subseteq \Gamma$,
\begin{equation*}\label{e:PCC_boundary_chart}
\widetilde{\blds{x}} = \widetilde{\psi}_p(\eta'), \quad \eta' = (\eta_1, \eta_2) \in \widetilde{U}_p.
\end{equation*}
The corresponding covariant basis and metric tensor are
\begin{equation}\label{e:PCC_boudary_basis}
\widetilde{\blds{q}}_i (\eta') = \dfrac{\pa \widetilde{\psi}_{p}}{\pa \eta_i}, \text{ } i =1,2,
\;
(\widetilde{q}_{ij} (\eta') )_{1 \leq i,j \leq 2}
    = (\widetilde{\blds{q}}_i \cdot \widetilde{\blds{q}}_j )_{1 \leq i,j \leq 2}
    = \text{diag}(\widetilde{\blds{q}}_1 \cdot \widetilde{\blds{q}}_1, \, \widetilde{\blds{q}}_2  \cdot \widetilde{\blds{q}}_2 ).
\end{equation}

Using \refE{x} with $\xi'$ and $\widetilde{\psi}$ replaced by $\eta'$ and $\widetilde{\psi}_p$, we define a chart $\psi_p$ from $U_p = \widetilde{U}_p \times (0, 3a)$ into $\Gamma_{3a}$ by
\begin{equation*}\label{e:PCC_chart}
\blds{x} = \psi_p(\blds{\eta}) = \widetilde{\psi}_p(\eta') - \xi_3 \blds{n}(\eta'), \quad \blds{\eta} = (\eta', \xi_3) \in U_p.
\end{equation*}
As before, $\xi_3$ is the distance from the boundary, which we note does not depend upon the choice of the boundary chart.

In the principal curvature coordinate system on $\widetilde{U}_p$, the unit outer normal vector $\blds{n}$ satisfies
\begin{equation*}\label{e:PCC_derivatives of normal vector}
\dfrac{\pa \n}{\pa \eta_i} = \kap_i(\eta') \widetilde{\blds{q}}_i, \text{ } i=1,2.
\end{equation*}
Hence, differentiating $\psi_p$ in the $\blds{\eta}$ variables gives the covariant basis of the coordinate system $\blds{\eta}$,
\begin{align}\label{e:PCC_basis}
	\blds{q}_i (\blds{\eta})
		&= (1 - \kap_i(\eta')\xi_3) \widetilde{\blds{q}}_i(\eta'), \; i = 1, 2,
	\quad
	\blds{q}_3 (\blds{\eta})
		= - \n(\eta').
\end{align}
Using \refEAnd{PCC_boudary_basis}{PCC_basis}, the metric tensor $(q_{ij})_{1 \leq i, j \leq 3}$ is written in the form
\begin{equation}\label{e:PCC_tensor}
    (q_{ij})_{1 \leq i,j \leq 3}
        = \left(
                \begin{array}{ccc}
                        (1 - \kap_1(\eta')\xi_3)^2 \,\widetilde{q}_{11}  &  0 &   0\\
                        0 &   (1 - \kap_2(\eta')\xi_3)^2 \, \widetilde{q}_{22}  &   0\\
                        0   & 0  & 1
                \end{array}
        \right)
\end{equation}
with its determinant, $q(\blds{\eta})$,  bounded away from zero. This is guaranteed by simply choosing the thickness, $3a > 0$, of the tubular neighborhood small enough.
\Ignore{ 
so that
\begin{equation}\label{e:Positivity of q_ii}
1 - \kap_i(\eta')\xi_3 > 0, \text{ } i=1,2, \text{ } 0 \leq \xi_3 \leq 3a.
\end{equation}
} 

It is easy to see that the coordinate system, $\blds{\eta}$, derived from the principal curvature coordinate system, satisfies \refEFirst{positivity of boundary tensor}\refENext{206}\refELast{positivity of tensor}. Hence we use the expression of the corrector $\tht^{\ep}$ , \refEFirst{221}\refENext{226}\refELast{227}, in $\blds{\eta}$ coordinates and write
\begin{equation*}\label{e:Corrector_intrinsic}
\tht^{\ep} = \tht^{\ep}_{\blds{\tau}} + \tht^{\ep}_{3} \blds{e}_3,
\end{equation*}
where
\begin{equation}\label{e:Corrector_eta}
\left\{
            \begin{array}{l}
                        \spacer \displaystyle
                        \tht^{\ep}_{\blds{\tau}} = - \ep \pdx{}{\xi_3} \pr{\sigma(\xi_3) 	\pr{1 -
                                    e^{-\frac{\xi_3}{\sqrt{\eps}}}}}
                                    \sum_{i=1}^{2} \Big\{ \frac{\sqrt{q_{ii}}}{\sqrt{q}}
                                    	\brac{\frac{\sqrt{q}}{\sqrt{q_{ii}}}}_{\xi_3 = 0}
			                        \widetilde{u}(\xi' ; t) \cdot \widehat{\blds{e}}_i\big|_{\xi_3=0}
			                        	\Big\} \widehat{\blds{e}}_i,\\
                        \displaystyle
                        \tht^{\ep}_{3} = \ep \sigma(\xi_3) \pr{1 - e^{-\frac{\xi_3}{\sqrt{\eps}}}}
                                    \frac{1}{\sqrt{q} } \Big\{\sum_{i=1}^2 \pdx{}{\xi_i} \Big( \frac{\sqrt{q}}
                                    	{\sqrt{q_{ii}}}\Big|_{\xi_3 = 0}
                                    \widetilde{u}(\xi' ; t) \cdot \widehat{\blds{e}}_i\big|_{\xi_3=0} \Big) \Big\} ,
            \end{array}
\right.
\end{equation}
for $\widehat{\blds{e}}_i = \blds{q}_i/ |\blds{q}_i|$, $i=1,2$, and $\blds{e}_3 = -\blds{n}$.

Using \refE{PCC_tensor}, it is easy to see that
\begin{equation}\label{e:scale_tan}
	\frac{\sqrt{q_{ii}}}{\sqrt{q}} \brac{\frac{\sqrt{q}}{\sqrt{q_{ii}}}}_{\xi_3 = 0}
            = \dfrac{1}{1 - \kap_j(\eta') \xi_3}
            \text{ for } i=1,2 \text{ and } j = 3-i.
\end{equation}
Then, combining \refEAnd{Corrector_eta}{scale_tan}, we find
\begin{equation}\label{e:intrinsic_tan}
	\tht^{\ep}_{\blds{\tau}}(\blds{\eta}) = - \ep \pdx{}{\xi_3} \pr{\sigma(\xi_3) 	\pr{1 -
                                    e^{-\frac{\xi_3}{\sqrt{\eps}}}}} \mathbf{M}(\blds{\eta}) \widetilde{u}(\eta'; t),
\end{equation}
where $\mathbf{M}$ is a smooth type $(1, 1)$ tensor defined, in our coordinates, by
\begin{equation}\label{e:M}
	\mathbf{M}(\blds{\eta}) \widetilde{F}
		:= \sum_{i=1}^2 \dfrac{1}{ 1 - \kappa_{3-i}(\eta') \xi_3 } \widetilde{F}_i \widehat{\blds{e}}_i ,
        \quad
	\widetilde{F} = \sum_{i=1}^{2} \widetilde{F}_i \widehat{\blds{e}}_i.
\end{equation}

\Ignore{ 
Expressed as above, it is clear that $\theta^\eps_{\blds{\tau}}$ depends only upon the principal curvatures, the principal directions, the distance, $\xi_3$, of a point in $\Gamma_{3a}$ from the boundary, and the given vector, $\widetilde{u}$. Hence, though the form of its expression depends on the coordinates chosen, its value at any point in $\Omega$ or $\Gamma$ does not.

It is perhaps not as readily apparent that this same can be said of the normal component, defined in \refE{Corrector_eta}$_2$. To see this, observe that $\theta^\eps_{\blds{n}} : = \tht^{\ep}_3 \blds{e}_3$ is the unique solution of the system
} 

On the boundary, the divergence operator is
\begin{align*}
	\dv_{\blds{\tau}}
		= \frac{1}{\sqrt{\widetilde{q}}} \sum_{i=1}^2 \pdx{}{\eta_i}
			\pr{\sqrt{\frac{\widetilde{q}}{\widetilde{q}_{ii}}} \, \widetilde{F}_i}.
\end{align*}
Then by \refE{PCC_tensor},
\begin{align*}
	\sqrt{q}
		= (1 - \kap_1(\eta')\xi_3) (1 - \kap_2(\eta')\xi_3) \sqrt{\widetilde{q}},
\end{align*}
so we can write $\theta_3^\eps$ in \refE{Corrector_eta} as
\begin{align}\label{e:intrinsic_normal}
	\theta_3^\eps
		= \ep \sigma(\xi_3) \frac{1 - e^{-\frac{\xi_3}{\sqrt{\eps}}}}
			{(1 - \kap_1(\eta')\xi_3) (1 - \kap_2(\eta')\xi_3)}
			\dv_{\blds{\tau}} \widetilde{u}.
\end{align}

%
%
\Ignore{
There are a number of equivalent ways to write $\theta^\eps_{\blds{n}} : = \tht^{\ep}_3 \blds{e}_3$. For one, $\theta^\eps_{\blds{n}}$ is the unique solution of the system
\begin{align*}
	\dv v = - \dv \theta^\eps_{\blds{\tau}} \text{ in } \Omega, \quad v = 0 \text{ on } \Gamma
\end{align*}
such that $v = v_3 \blds{e}_3$.
On the other hand, using \refE{211}, we see that
\begin{align*}
	\dv (\theta^\eps_{3}\blds{e}_3)
		= \frac{1}{\sqrt{q}} \pdx{}{\xi_3} (\sqrt{q} \theta^\eps_3).
\end{align*}
Hence, $ \sqrt{q} \, \theta_3^\eps$ is the unique solution to
\begin{align*}
	\displaystyle \pdx{}{\xi_3} \brac{\sqrt{q} \theta^\eps_3}
		&= - \sqrt{q} \dv \theta^\eps_{\blds{\tau}},
	\quad
	\theta^\eps_{3} = 0 \text{ on } \Gamma;
\end{align*}
namely,
\begin{align}\label{e:intrinsic_normal}
	\begin{split}
	\theta^\eps_3(\blds{\eta})
		&= - \frac{1}{\sqrt{q}(\blds{\eta})} \int_0^{\xi_3}  \sqrt{q}(\eta', s)
				\dv \theta^\eps_{\blds{\tau}}(\eta', s) \, ds \\
		&= - \int_0^{\xi_3}  \frac{\pr{1 - \kappa_1 s} \pr{1 - \kappa_2 s}}
			{\pr{1 - \kappa_1 \xi_3} \pr{1 - \kappa_2 \xi_3}}
				\dv \theta^\eps_{\blds{\tau}}(\eta', s) \, ds,
	\end{split}
\end{align}
where we used \refE{PCC_tensor}.
} 
\Ignore{ 
\begin{align*}
	\frac{\sqrt{q}(\eta', s)}{\sqrt{q}(\blds{\eta})}
		= \frac{\pr{1 - \kappa_1 s} \pr{1 - \kappa_2 s}}
			{\pr{1 - \kappa_1 \xi_3} \pr{1 - \kappa_2 \xi_3}}.
\end{align*}
} 


Although we assumed in its derivation that we were near a non-umbilical point so that we could construct a principal curvature coordinate system, our expression for $\theta^\eps$ is perfectly valid at an umbilical point, where we simply have $\kappa_1 = \kappa_2$ thanks to the smoothness of the curvatures in the tangential variables.

Finally, a straightforward but lengthy calculation, which we omit, shows that \refEAnd{226}{227} transforms to \refEAnd{intrinsic_tan}{intrinsic_normal} under the change of variables from $\psi$ to $\psi_p$, showing that our corrector in the form \refEAnd{226}{227} is coordinate-independent. (This is perhaps not immediately obvious, because \refEAnd{226}{227} involve the metrics both on the boundary and in the tubular neighborhood.)

\begin{remark}\label{R:On PCC}
For most smooth, bounded domains in $\R^3$, principal curvature coordinate systems can be constructed in the neighborhood of all but at most isolated points; in fact, having only isolated umbilical points is (in some sense) generic. And, for instance, a sphere, while it consists only of umbilical points, can be covered by two charts, both of which use principal curvature coordinates (essentially, spherical coordinates). When such coordinates suit the boundary of a domain, the expression for the tangential corrector in \refEAnd{intrinsic_tan}{M} is both simpler to calculate and simpler to interpret than the expression in \refE{226}. In such coordinates, the expressions for the differential operators, such as $\dv$, $\curl$, $\Delta$, can also be written more simply. Though we used principal curvature coordinates in this section to prove that our corrector is independent of the choice of charts, we cannot restrict ourselves to such coordinates in the rest of our analysis, as that would put constraints on the geometry of the domains that we would be able to treat.
\end{remark}

\begin{remark}\label{e:torus}
As a special example of a smooth bounded domain in $\R^3$, consider a solid torus, which has no umbilical points on its boundary. Hence, we need only one principal curvature coordinate system, if we allow it to be periodic in the tangential variables. In this sense, the solid torus is the simplest smooth bounded domain to work with in $\R^3$.
\end{remark}

\subsection{Bounds on the corrector}\label{S:CorrectorBoundsCurved}
We follow the convention described at the beginning of \refS{ChannelCorrectorBounds}, though now the tangential variables are $\xi_1$ and $\xi_2$, and as in \refE{olal}, we let
$
	\ol{\al} = \norm{\Cal{A}}_{C^m(\Gamma)},
$
$m>6$.
Then, from \refEAnd{226}{227}, we infer that
\begin{equation}\label{e:230}
	\norm{\dfrac{\pa \tht^{\ep}_i}{\pa \tau}}_{L^{\infty}([0,T]\times\ov{U})}
		\le \kap_T(1+ \ov{\alpha})\ep^{\frac{1}{2}},
	\quad
	\norm{\dfrac{\pa \tht^{\ep}_i}{\pa \xi_3}}_{L^{\infty}([0,T]\times\ov{U})}
		\le \kap_T(1+ \ov{\alpha}), \text{ } i=1,2,
\end{equation}
and
\begin{equation}\label{e:231}
	\norm{\dfrac{\pa \tht^{\ep}_3}{\pa \tau}}_{L^{\infty}([0,T]\times\ov{U})}
		\le \kap_T(1+ \ov{\alpha})\ep,
	\quad
	\norm{\dfrac{\pa \tht^{\ep}_3}{\pa \xi_3}}_{L^{\infty}([0,T]\times\ov{U})}
		\le \kap_T(1+ \ov{\alpha})\ep^{\frac{1}{2}}.
\end{equation}

We  now state the estimates on the corrector in the lemma below, which we omit the proof as it is essentially the same as that of \refL{2} because of \refE{positivity of tensor}.
\begin{lemma}\label{L:3}
Assume \refE{0-6} holds and that $k, l, n \ge 0$ are integers either $l=1$, $k=0$ or $l=0$, $0 \leq k \leq 2$. Then the corrector, $\tht^{\ep}$, defined by \refEAnd{226}{227}, satisfies
\begin{align*}\label{e:232}
	\left\{
	\begin{array}{l}
	\spacer \norm{
			\dfrac{\pa^{l+k+n} \tht^{\ep}_i}{\pa t^l \pa \tau^k \pa \xi_3^n}}_{L^{\infty}(0,T;L^{2}(U))}
		\le C |V|^{\frac{1}{2}} (1 + \ov{\alpha})\ep^{\frac{3}{4} - \frac{n}{2}}  , \text{ } i=1,2, \\
    \spacer
    \norm{   \dfrac{\pa^{l+k} \tht^{\ep}_3}
			{\pa t^l \pa \tau^k }}_{L^{\infty}(0,T;L^{2}(U))}
		\le C |V|^{\frac{1}{2}} (1 + \ov{\alpha}) \ep,\\
    \norm{   \dfrac{\pa^{l+k+n+1} \tht^{\ep}_3}
			{\pa t^l \pa \tau^k \pa \xi_3^{n+1}}}_{L^{\infty}(0,T;L^{2}(U))}
		\le C |V|^{\frac{1}{2}} (1 + \ov{\alpha}) \ep^{\frac{3}{4} - \frac{n}{2}},
	\end{array}
	\right.
\end{align*}
for $C = C(T, l, k, n, u_0, f)>0$, independent of $\ep$, $\mathcal{A}$ and the measure $|V|$ of  $V = \psi(U)$.
\end{lemma}

In addition to the estimates in Lemma \ref{L:3}, as for the case of the channel domain, one can easily verify that the corrector $\tht^{\ep}$ satisfies that, $i=1,2$,
\begin{equation}\label{e:special_est_General}
\norm{ \dfrac{\xi_3}{\sqrt{\ep}} \dfrac{\pa \tht^{\ep}_i}{\pa \xi_3}}_{L^{\infty}(0,T;L^{2}(U))} \leq C |V|^{\frac{1}{2}} (1 + \ov{\alpha})\ep^{\frac{1}{4}}  ,
\quad
\norm{ \dfrac{\xi_3}{\sqrt{\ep}} \dfrac{\pa \tht^{\ep}_3}{\pa \xi_3}}_{L^{\infty}(0,T;L^{2}(U))} \leq C |V|^{\frac{1}{2}} (1 + \ov{\alpha})\ep^{\frac{1}{2}}.
\end{equation}

\subsection{Error analysis}\label{S:S3.2}
We set the remainder:
\begin{equation}\label{e:240}
w^{\ep} : = u^{\ep } - u^0 - \tht^{\ep}.
\end{equation}
Then, using \refEFirst{0-1}\refENext{Euler}\refENext{General_NBC}\refELast{240} and the fact that $\tht^{\ep}_3 =0$ on $\Gamma$, the equations for $w^{\ep}$ read
\begin{equation}\label{e:241}
	\left\{
	\begin{array}{l}
        \spacer
		\dfrac{\pa w^{\ep}}{\pa t} -\ep \Delta w^{\ep} + \nabla \big(p^{\ep}-p^0\big)
			= \ep \Delta u^0 + R_{\ep}(\tht^{\ep}) - J_{\ep}(u^\ep, u^0),
				\text{ in } \Omega \times(0,T), \\
		\spacer
        \text{div }w^{\ep}
			= 0, \text{ in } \Omega  \times(0,T), \\
		\spacer
		w^{\ep} \cdot \n
			= 0, \text{ on } \Gamma, \\
		\spacer
        \big[ \mathbf{S}(w^{\ep}) \n \big]_{\text{tan}} + \mathcal{A} w^{\ep}
			= - \big[ \mathbf{S}(u^0 + \tht^{\ep}) \n \big]_{\text{tan}}
                - \mathcal{A} (u^0 + \tht^{\ep}), \text{ on } \Gamma, \\
		w^{\ep}|_{t=0}
			= -\tht^{\ep}| _{t=0}, \text{ in } \Omega,
	\end{array}
	\right.
\end{equation}
where $R_{\ep}(\cdot)$ and $J_{\ep}(\cdot, \cdot)$ are defined by \refE{42} and \refE{43}.

We multiply the equation \refESub{241}{1} by $w^{\ep}$, integrate it over $\Omega$ and then use \refL{1}.  After applying the Schwarz and Young inequalities to the right-hand side of the resulting equation, we find:
\begin{equation}\label{e:244}
	\begin{split}
		\dfrac{d}{d t} &\norm{w^{\ep}}_{L^2(\Omega)}^2
			+ 4 \ep \norm{\mathbf{S}( w^{\ep})}_{L^2(\Omega)}^2 \\
		&\le \ep^2\norm{\Delta u^0}_{L^2(\Omega)}^2
			+ \norm{R_{\ep}(\tht^{\ep})}_{L^2(\Omega)}^2
			+ 2 \norm{w^{\ep}}_{L^2(\Omega)}^2 \\
		& - 4 \ep \int_{\Gamma} \Big( \mathcal{A} w^{\ep} + \big[ \mathbf{S}(u^0 + \tht^{\ep}) \n
                \big]_{\text{tan}} + \mathcal{A} (u^0 + \tht^{\ep}) \Big) \cdot  w^{\ep} \, d S\\
		& - 2 \int_{\Omega} J_{\ep}(u^{\ep}, u^0) \cdot w^{\ep} \, d \blds{x}.
	\end{split}
\end{equation}
To go further, using the Korn inequality, we first notice that
\begin{equation}\label{e:245}
	\norm{\mathbf{S}(w^{\ep})}_{L^2(\Omega)}^2
		\ge \kap_{\mathbf{S}}\big\{ \norm{\nabla w^{\ep}}_{L^2(\Omega)}^2
			+ \norm{w^{\ep}}_{L^2(\Omega)}^2 \big\}
		\ge \kap_{\mathbf{S}}\norm{\nabla w^{\ep}}_{L^2(\Omega)}^2
\end{equation}
for a constant, $\kap_{\mathbf{S}}$, depending on the domain, but independent of $\ep$ and $\alpha$.

Restricted to the range, $V$, of any chart, $\psi$, we find, using \refE{42} with $v$ replaced by $\tht^{\ep}$ and \refEFirst{214}\refELast{215} for each $\tht^{\ep}$, that
\begin{align*}
	\norm{R_{\ep}(\tht^{\ep})}_{L^2(V)}^2
		&\le \Big\| \dfrac{\pa \tht^\eps}{\pa t} \Big\|_{L^2(U)}^2
			+ \ep \sum_{i=1}^{3}\Big\| \mathcal{S}^i \tht^\eps + \mathcal{L}^i \tht^\eps_i
			+ \dfrac{\pa^2 \tht^\eps_i}{\pa \xi_3^2} \Big\|_{L^2(U)}^2 \\
		&\le \kap_T (1+ \ov{\alpha}^2) \ep^{\frac{3}{2}},
\end{align*}
where we also used \refR{201} and \refL{3}. Since we have a finite number of charts on $\Gamma_{3a}$, and since $\theta^\eps$ is supported in $\Gamma_{3 a}$ by \refE{229}, the same estimate holds on $\Omega$; namely,
\begin{align}\label{e:249}
	\norm{R_{\ep}(\tht^{\ep})}_{L^2(\Omega)}^2
		\le \kap_T (1+ \ov{\alpha}^2) \ep^{\frac{3}{2}}.
\end{align}

To estimate the fourth term in the right-hand side of \refE{244}, we write
\begin{equation}\label{e:B_Int_1}\begin{array}{l}
    \spacer \displaystyle
    \Big| 4 \ep \int_{\Gamma} \Big( \mathcal{A} w^{\ep} + \big[ \mathbf{S}(u^0 + \tht^{\ep}) \n
                \big]_{\text{tan}} + \mathcal{A} (u^0 + \tht^{\ep}) \Big)  \cdot  w^{\ep} \, d S\Big|\\
    \quad \quad \displaystyle
    \le \kap_T \ep \ov{\al} \| w \|^2_{L^2(\Gamma)} + \kap_T \ep \Big\|  \big[ \mathbf{S}(u^0 + \tht^{\ep}) \n
                \big]_{\text{tan}} + \mathcal{A} (u^0 + \tht^{\ep}) \Big\|_{L^2(\Gamma)} \| w \|_{L^2(\Gamma)}.
\end{array}
\end{equation}
On each $\widetilde{V} \subset \Gamma$, the rage of the boundary chart, $\widetilde{\psi}$, using \refEFirst{219-9}\refENext{219-10}\refENext{224}\refELast{223Approx}, we have
\begin{equation*}
    \begin{array}{l}
	        \spacer
            \Big( \big[ \mathbf{S}(u^0 + \tht^{\ep}) \n \big]_{\text{tan}} + \mathcal{A} (u^0 + \tht^{\ep}) \Big)\Big|_{\widetilde{V}}\\
            \spacer \displaystyle
            = \sum_{i=1}^{2} \brac{ \dfrac{1}{2} \widetilde{u}_i - \dfrac{1}{2} \dfrac{\pa \tht^\eps_i}{\pa \xi_3}
			+ \mathcal{M}_i\Big( \tht^\eps_i, \dfrac{\pa \tht^{\ep}_3}{\pa \xi_i} \Big) + \sum_{j=1}^2 \alpha_{ij}
            \tht^{\ep}_j}_{\xi_3 = 0} \blds{e}_i |_{\xi_3 = 0} \\
            \displaystyle
            = \sum_{i=1}^{2} \brac{\mathcal{M}_i\Big( \tht^\eps_i, \dfrac{\pa \tht^{\ep}_3}{\pa \xi_i} \Big) + \sum_{j=1}^2
            \alpha_{ij} \tht^{\ep}_j + \dfrac{1}{2} \sqrt{\ep} E (\xi_1, \xi_2; t)}_{\xi_3 = 0} \blds{e}_i |_{\xi_3 = 0},
    \end{array}
\end{equation*}
where $\widetilde{u}_i$ and $E$ are defined by \refEAnd{224}{Error_Boundary}. Since this bound holds for all charts, using Remark \ref{R:NavierBC} and \refEFirst{219-10}\refENext{230}\refELast{231}, we find that
\begin{equation}\label{e:B_Int_3}
\ep \Big\| \big[ \mathbf{S}(u^0 + \tht^{\ep}) \n
                \big]_{\text{tan}} + \mathcal{A} (u^0 + \tht^{\ep}) \Big\|_{L^2(\Gamma)} \leq \kap_T (1 + \ov{\al}^2) \ep^{\frac{3}{2}}.
\end{equation}
Thanks to \refEFirst{245}\refENext{249}\refENext{B_Int_1}\refELast{B_Int_3}, applying \refL{trace-like} and the Poincar\'{e} inequality, (\ref{e:244}) yields that
%
%
%
%
%
%
\begin{align}\label{e:252}
	\begin{split}
		\dfrac{d}{d t}&\norm{w^{\ep}}_{L^2(\Omega)}^2
			+ 2 \kap_{\mathbf{S}} \ep \norm{\nabla w^{\ep}}_{L^2(\Omega)}^2 \\
		&\le  \kap_T (1 + \ov{\alpha}^4) \ep^{\frac{3}{2}} + \kap_T(1 + \ov{\alpha}^2 ) \norm{w^{\ep}}_{L^2(\Omega)}^2
			- 2 \int_{\Omega} J_{\ep}(u^{\ep}, u^0) \cdot w^{\ep} \, d \blds{x}.
	\end{split}
\end{align}

To estimate the last term of \refE{252}, using \refE{53}, we write
\begin{align*}
	\int_{\Omega} J_{\ep}(u^{\ep}, u^0) \cdot w^{\ep} \, d \blds{x}
		:= \sum_{j=1}^5 \mathcal{J}_{\ep}^j ,
\end{align*}
where $\mathcal{J}_{\ep}^j $, $1\leq j \leq 5$, are given by \refE{55}.
Due to \refEFirst{229}\refENext{230}\refELast{231} and \refL{3}, one can easily verifies that $\mathcal{J}_{\ep}^j $, $j= 2,3,$ satisfies the same estimate, appearing in \refEAnd{60}{61}, as for the case of a channel domain. That is,
\begin{align*}
	\Big|  \sum_{j=1}^{3} \mathcal{J}_{\ep}^j
		\Big| \leq \kap_T (1+ \ov{\alpha}^2 ) \ep^{\frac{3}{2}} + \kap_T(1+ \ov{\alpha})
\norm{w^{\ep}}_{L^2(\Omega)}^2.
\end{align*}

To bound $\mathcal{J}_{\ep}^4$ and  $\mathcal{J}_{\ep}^5$, it is sufficient to work in a single chart, $\psi \colon U \to V$, as there a finite number, $N$,  of them, which just introduces the constant, $N$.

For $\mathcal{J}_{\ep}^4$, using \refEFirstSub{55}{4}\refELast{229}, we write
\begin{align*}
	\abs{\mathcal{J}_{\ep}^4}
		&\le \norm{(u^0\cdot\nabla)\tht^{\ep}}_{L^2(V)}  \norm{w^{\ep}}_{L^2(\Omega)} \\
		&\le (\text{using \refEAnd{218}{219} with $F$, $G$ replaced by $u^0$, $\tht^{\ep}$,
			respectively}) \\
		&\le \kap_T \norm{u^0}_{L^{\infty}(\Omega)}
			\sum_{i=1}^{3} \Big\{ \norm{\tht^\eps_i}_{L^2(U)}
			+ \sum_{j=1}^{2} \norm{\dfrac{\pa \tht^{\ep}_i}{ \pa \xi_j}}_{L^2(U)} \Big\}
			\norm{w^{\ep}}_{L^2(\Omega)}\\
			& \quad + \kap_T \sqrt{\ep} \norm{\dfrac{u^0 \cdot \blds{e}_3}{\xi_3}}_{L^{\infty}(\Gamma_{2a})}
                    		\sum_{i=1}^{3}
				\norm{\dfrac{\xi_3}{\sqrt{\ep}} \dfrac{\pa \tht^{\ep}_i}{\pa \xi_3}}_{L^{2}
				(U)} \norm{w^{\ep}}_{L^2(\Omega)} \\
		&\le (\text{using \refE{special_est_General}, \refL{3} and the regularity of
			$u^0$ with $(u^0 \cdot \blds{e}_3)|_{\xi_3 = 0} = 0$}) \\
		&\le \kap_T (1 + \ov{\alpha}) \ep^{\frac{3}{4}} \norm{w^{\ep}}_{L^2(\Omega)}
		\le \kap_T (1 + \ov{\alpha}^2) \ep^{\frac{3}{2}} + \norm{w^{\ep}}_{L^2(\Omega)}^2.
\end{align*}

For $\mathcal{J}_{\ep}^5$, using \refEAnd{218}{219} with $G$ and $F$ replaced by $\tht^{\ep}$, and using \refE{229}, we find
\begin{align*}
	\abs{\mathcal{J}_{\ep}^5}
		&\le  \norm{(\tht^{\ep} \cdot \nabla)\tht^{\ep}}_{L^2(V)}
			\norm{w^{\ep}}_{L^2(\Omega)} \\
		&\le \kap_T  \norm{\tht^\ep}_{L^{\infty}(U)}  \sum_{i = 1}^{3}
			\Big\{ \norm{\tht^\eps_i}_{L^2(U)} + \sum_{j=1}^2 \norm{\dfrac{\pa \tht^\eps_i}{\pa \xi_j}}_{L^2(U)}
            \Big\}\norm{w^{\ep}}_{L^2(\Omega)} \\
		&\quad
			+  \kap_T \norm{\tht^\eps_3}_{L^{\infty}(U)} \sum_{i=1}^{3}
			\norm{\dfrac{\pa \tht^\eps_i}{\pa \xi_3}}_{L^2(U)}
			\norm{w^{\ep}}_{L^2(\Omega)} \\
		&\le  (\text{using \refEAnd{230}{231} and \refL{3}}) \\
		&\le \kap_T (1 + \ov{\alpha}^2) \ep^{\frac{5}{4}} \norm{w^{\ep}}_{L^2(\Omega)}
			\le \kap_T (1 + \ov{\alpha}^4) \ep^{\frac{5}{2}} + \norm{w^{\ep}}_{L^2(\Omega)}^2.
\end{align*}
Using these bounds on $\mathcal{J}_\eps^i$, $1 \leq i \leq 5$, \refE{252} becomes
\begin{align*}
\dfrac{d}{d t}\norm{w^{\ep}}_{L^2(\Omega)}^2 +  2 \kap_{\mathbf{S}} \ep |\nabla w^{\ep}|_{L^2(\Omega)}^2
                    \leq \kap_T (1+\ov{\alpha}^4)\ep^{\frac{3}{2}} + \kap_T (1 + \ov{\alpha}^2) \norm{w^{\ep}}_{L^2(\Omega)}^2.
\end{align*}
Moreover, using \refEFirst{226}\refENext{227}\refELastSub{241}{5}, we see that
\begin{align*}
\normmark  w^{\ep}|_{t=0} \normmark _{L^2(\Omega)} = \norm{\tht^{\ep}|_{t=0}}_{L^2(\Gamma_{2a})} \leq \kap_T (1 + \ov{\al}) \ep^{\frac{1}{2}} \smallnorm{e^{-\frac{\xi_3}{\sqrt{\ep}}}}_{L^2(\Gamma_{2a})} + l.o.t. \leq \kap_T (1 + \ov{\al}) \ep^{\frac{3}{4}}.
\end{align*}
Thanks to the Gronwall inequality, we finally have the estimates,
\begin{align}\label{e:267}\
	\norm{w^{\ep}}_{L^{\infty}(0,T; L^2(\Omega))}
		\le \kap(T, \ov{\al}, u_0, f) \ep^{\frac{3}{4}},
	\quad
	\norm{w^{\ep}}_{L^{2}(0,T; H^1(\Omega))}
		\le \kap(T, \ov{\al}, u_0, f) \ep^{\frac{1}{4}}.
\end{align}

\subsection{Proof of convergence}\label{S:S3.3}
Using \refE{240}, we first notice that
\begin{align*}
|u^{\ep} - u^0| \leq |w^{\ep}| + |\tht^{\ep}|, \text{ } \text{ (pointwise in $\Omega \times (0,T)$)}.
\end{align*}
Then, using \refE{267} and \refL{3}, \refE{68} follows. \hspace{30mm}$\Box$

%
%
\section{Uniform convergence}\label{S:S-4}

\noindent With the estimates we now have, the proof of \refE{uniform} is quite simple.

Because we assume that $m > 6$ in \refE{0-6}, $u_0 \in E^m\cap H$ and (by Sobolev embedding) $\nabla u_0 \in W^{1, \infty}_{co}$. Hence, both \refE{0-7} and the hypotheses for \refT{MR} hold\footnote{The hypotheses in \refT{1} are not the minimal ones insuring this.}, so we can use \refEAnd{0-7}{MR1} to conclude that
\begin{align*}
	\norm{u^\eps - u^0}_{L^\iny(0, T; H_{co}^m(\Omega))},
		\qquad
		\norm{\grad(u^\eps - u^0)}_{L^\iny(0, T; H_{co}^{m - 1}(\Omega))}
			\le C.
\end{align*}
Then using \refE{68}, \refT{AnistropicAgmonsInequality}, and \refR{AgmonsChannel}, \refE{uniform} follows.

\begin{remark}\label{R:LipschitzInterpolation}
	Since also $u^\eps - u^0$ lies in $L^\iny([0, T]; W^{1, \iny})$ by \refT{MR}, we could use the
	Gagliardo-Nirenberg interpolation inequality (see, for instance, p. 314 of \cite{Brezis2011}),
	\begin{align*}
		\norm{u^\eps - u^0}_{L^\iny}
			\le C \norm{u^\eps - u^0}_{L^2}^{\frac{2}{5}} \norm{u^\eps - u^0}_{W^{1, \iny}}^{\frac{3}{5}},
	\end{align*}
	to give $\norm{u^\eps - u^0}_{L^\iny} \le C \eps^{3/10}$. This is the same rate that is obtained
	for $m = 6$ in \refE{uniform}; since, however, we require that $m > 6$, \refE{uniform} always gives
	a better rate than this.
\end{remark}

\appendix

%
%
\section{An anisotropic Agmon's inequality}\label{S:Agmons}

\noindent In this section we develop a version of Agmon's inequality in $d$ dimensions, $d = 2$ or $3$, that is suitable for applying to anisotropic problems in which there is more control over tangential (horizontal) derivatives than over normal (vertical) derivatives.

We use the notation $A \ll B$ to mean that $A \le CB$ for some constant, $C$, which may depend upon the geometry of an underlying domain but not upon anything else. If $C$ depends on some parameter, $m$, then we write $A \ll_m B$.

Our starting point is the following simple lemma:

\begin{lemma}\label{L:Agmon}
	Let $U$ be a bounded domain in $\R^d$, $d = 1, 2, 3$. For any $f$ in $H^k(U)$, $k \ge d$,
	\begin{align*}
		\norm{f}_{L^\iny(U)}
			\ll_k \norm{f}_{L^2(U)}^{1 - \frac{1}{2k}} \norm{f}_{H^k(U)}^{\frac{1}{2k}} &\text{ if } d = 1,
							\\
		\norm{f}_{L^\iny(U)}
			\ll_k \norm{f}_{L^2(U)}^{1 - \frac{1}{k}} \norm{f}_{H^k(U)}^{\frac{1}{k}} &\text{ if } d = 2,
							\\
		\norm{f}_{L^\iny(U)}
			\ll_k \norm{f}_{L^2(U)}^{1 - \frac{3}{2k}} \norm{f}_{H^k(U)}^{\frac{3}{2k}} &\text{ if } d = 3.
	\end{align*}
\end{lemma}
\begin{proof}
	Combine the 1D Agmons' inequality, $\norm{f}_{L^\iny(U)}
	\ll \norm{f}_{L^2(U)}^{1/2	} \norm{f}_{H^1(U)}^{1/2}$,
	2D Agmons' inequality, $\norm{f}_{L^\iny(U)}
	\ll \norm{f}_{L^2(U)}^{1/2} \norm{f}_{H^2(U)}^{1/2}$,
	or 3D Agmon's inequality,
	$\norm{f}_{L^\iny(U)}
	\ll \norm{f}_{L^2(U)}^{1/4} \norm{f}_{H^2(U)}^{3/4}$
	with the Sobolev interpolation inequality,
	$\norm{f}_{H^j(U)} \ll_k \norm{f}_{L^2(U)}^{1 - j/k} \norm{f}_{H^k(U)}^{j/k}$,
	$0 \le j \le k$.
\end{proof}

\begin{theorem}\label{T:AnistropicAgmonsInequality}
	Let $\Omega$ be a bounded domain in $\R^3$ with $C^{m + 1}$-boundary, $m \ge 3$, and
	let $\Gamma_a$ be the tubular neighborhood of fixed width $a > 0$ interior to $\Omega$.
	Suppose that $f$ and $\grad f$ lie in the space $H^m_{co}(\Omega)$ of \refD{HConormal}. Then
	\begin{align*}
		\norm{f}_{L^\iny(\Gamma_a)}
			&\ll_{m, a}
				\norm{f}_{L^2(\Omega)}^{\frac{1}{2} - \frac{1}{2m}}
					\norm{f}_{H^m_{co}(\Omega)}^{\frac{1}{2m}}
					\brac{\norm{f}_{L^2(\Omega)}+ \norm{\grad f}_{H^m_{co}(\Omega)}}^{\frac{1}{2}}, \\
		\norm{f}_{L^\iny(\Omega \setminus \Gamma_a)}
			&\ll_{m, a}
				\norm{f}_{L^2(\Omega)}^{1 - \frac{3}{2m}}
					\norm{f}_{H^m_{co}(\Omega)}^{\frac{3}{2m}}.
	\end{align*}
\end{theorem}
\begin{proof}
We define the chart, $\psi$, as in the beginning of \refS{S-3}. In this chart, we can define a local conormal basis, $(X_1, X_2, X_3)$, by $X_i f(\blds{x}) = \prt_i (f \circ \psi)(\psi^{-1}(\blds{x}))$, $i = 1, 2$ and
$
	X_3 f(\blds{x}) = \frac{\psi^{-1}_3}{1 + \psi^{-1}_3} \prt_3 (f \circ \psi)(\psi^{-1}(\blds{x})).
$
We will have need, however, only for $X_1$ and $X_2$.
\Ignore { 
So, for instance, if $p = \psi(\xi_1, \xi_2, \xi_3)$ lies in $V$ then $\psi_3^{-1}(p) = \xi_3$ and
\begin{align*}
	X_3 f (p) = \frac{xi_3}{1 + \xi_3} \prt_3 \psi(\psi^{-1}(p)) \cdot \grad f (p),
\end{align*}
showing that $X_3$, and similarly $X_1$ and $X_2$, are derivations that are clearly tangent to the boundary and generate all tangent vector fields locally on $V$.
} 

It suffices to assume that $f$ lies in $C^\iny(\ol{\Omega})$.
Restricting ourselves to the one chart, $\psi$, by \refL{Agmon}, we have, for any $\blds{\xi} = (\xi_1, \xi_2, \xi_3)$ in $U$,
\begin{align*}
	f \circ \psi(\blds{\xi})
		&\ll_m \norm{f \circ \psi (\cdot, \cdot, \xi_3)}_{L^2(U_0)}^{1 - \frac{1}{m}}
		\norm{f \circ \psi (\cdot, \cdot, \xi_3)}_{H^m(U_0)}^{\frac{1}{m}}.
\end{align*}
Applying \refL{Agmon} again, this time in 1D with $k = 1$, gives
\begin{align*}
	&\norm{f \circ \psi(\cdot, \cdot, \xi_3)}_{L^2(U_0)}^2
		= \int_{U_0} f \circ \psi(\xi_1', \xi_2', \xi_3)^2 \, d\xi_1' \, d\xi_2' \\
		&\qquad
			\ll \int_{U_0}
				\norm{f \circ \psi(\xi_1', \xi_2', \cdot)}_{L^2(0, a)}
				\norm{f \circ \psi(\xi_1', \xi_2', \cdot)}_{H^1(0, a)} \, d\xi_1' \, d\xi_2' \\
		&\qquad
		\le \pr{\int_{U_0} \norm{f \circ \psi(\xi_1', \xi_2', \cdot)}_{L^2(0, a)}^2 \, d\xi_1' \, d\xi_2'}^{\frac{1}{2}}
			\pr{\int_{U_0} \norm{f \circ \psi(\xi_1', \xi_2', \cdot)}_{H^1(0, a)}^2 \, d\xi_1' \, d\xi_2'}^{\frac{1}{2}} \\
		&\qquad
		= \pr{\int_U \abs{f \circ \psi(\blds{\xi})}^2 \, d \boldsymbol{\xi}}^{\frac{1}{2}}
			\pr{\int_U \brac{\abs{f \circ \psi(\blds{\xi})}^2
				+ \abs{\grad f(\psi(\blds{\xi}) \cdot \prt_3 \psi(\blds{\xi})}^2} \, d \boldsymbol{\xi}}^{\frac{1}{2}} \\
		&\qquad
		\ll_a \pr{\int_U \abs{f \circ \psi(\blds{\xi})}^2 \abs{J(\blds{\xi})} \, d \blds{\xi}}^{\frac{1}{2}}
			\pr{\int_U \brac{\abs{f \circ \psi(\blds{\xi}) }^2
				+ \abs{\grad f(\psi(\blds{\xi}))}^2} \abs{J(\blds{\xi})} \, d \blds{\xi}}^{\frac{1}{2}} \\
		&\qquad
		= \norm{f}_{L^2(V)} \norm{f}_{H^1(V)}.
\end{align*}
The second $\ll$ followed because the magnitude of the Jacobian determinant, $J$, is bounded away from zero and $\prt_3 \psi$ is bounded above. Because there are a finite number of charts on $\Gamma_a$, the bounds are uniform over $\Gamma_a$.

Similarly, for all multiindices, $\al = (\al_1, \al_2, 0)$ with $\abs{\al} \le m$, applying \refL{Agmon} with $k = 1$ gives
\begin{align*}
	&\norm{D^\al(f \circ \psi)(\cdot, \cdot, \xi_3)}_{L^2(U_0)}^2 \\
		&\qquad
		\ll_{m, a} \pr{\int_U \abs{D^\al(f \circ \psi)}^2}^{\frac{1}{2}}
			\pr{\int_U \brac{\abs{D^\al(f \circ \psi)}^2
				+ \abs{\prt_3 D^\al(f \circ \psi)}^2}}^{\frac{1}{2}} \\
		&\qquad
		= \pr{\int_V \abs{X^\al f}^2}^{\frac{1}{2}}
			\pr{\int_V \abs{X^\al f}^2
				+ \int_U \abs{D^\al \prt_3(f \circ \psi)}^2}^{\frac{1}{2}} \\
		&\qquad
		\ll \pr{\int_V \abs{X^\al f}^2}^{\frac{1}{2}}
			\pr{\int_V \abs{X^\al f}^2
				+ \int_V \abs{X^\al \grad f}^2}^{\frac{1}{2}} \\
		&\qquad
		= \norm{X^\al f}_{L^2(V)}\norm{X^\al f}_{H^1(V)}.
\end{align*}
But this is true for all $\abs{\al} \le m$, so
\begin{align*}
	\norm{f \circ \psi (\cdot, \cdot, \xi_3)}_{H^m(U_0)}
		\le
		\norm{f}_{H^m_{co}(V)}^{1/2}
		\brac{\norm{f}_{L^2(\Omega)} + \norm{\grad f}_{H^m_{co}(V)}}^{1/2},
\end{align*}
where we can use the conormal Sobolev space since the only derivative in the normal direction occurs in $\grad f$ itself. Also, because $\Gamma$ is $C^{m + 1}$, $\psi$ can be chosen to be $C^{m + 1}(\Gamma_a)$ and hence $f \circ \psi$ has sufficient smoothness.

Combining these bounds we have,
\begin{align*}
	\norm{f}_{L^\iny(V)}
		&\ll_{m, a}
			\norm{f}_{L^2(V)}^{\frac{1}{2} - \frac{1}{2m}}
			\norm{f}_{H^1(V)}^{\frac{1}{2} - \frac{1}{2m}}
			\norm{f}_{H^m_{co}(V)}^{\frac{1}{2m}}
			\brac{\norm{f}_{L^2(\Omega)} + \norm{\grad f}_{H^m_{co}(V)}}^{\frac{1}{2m}} \\
		&\le
			\norm{f}_{L^2(V)}^{\frac{1}{2} - \frac{1}{2m}}
			\norm{f}_{H^m_{co}(V)}^{\frac{1}{2m}}
			\brac{\norm{f}_{L^2(\Omega)} + \norm{\grad f}_{H^m_{co}(V)}}^{\frac{1}{2}}.
\end{align*}
Summing over all the $V$ gives
\begin{align*}
	\norm{f}_{L^\iny(\Gamma_a)}
		&\ll_{m, a}
			\norm{f}_{L^2(\Omega)}^{\frac{1}{2} - \frac{1}{2m}}
			\norm{f}_{H^m_{co}(\Omega)}^{\frac{1}{2m}}
			\brac{\norm{f}_{L^2(\Omega)} + \norm{\grad f}_{H^m_{co}(V)}}^{\frac{1}{2}}.
\end{align*}
With $W = \Omega \setminus \Gamma_a$, \refL{Agmon} gives
\begin{align*}
	\norm{f}_{L^\iny(W)}
			\ll_m \norm{f}_{L^2(W)}^{1 - \frac{3}{2m}} \norm{f}_{H^m(W)}^{\frac{3}{2m}}
			\ll_a \norm{f}_{L^2(\Omega)}^{1 - \frac{3}{2m}}
			\norm{f}_{H^m_{co}(\Omega)}^{\frac{3}{2m}}.
\end{align*}
Combining these last two inequalities completes the proof.
\end{proof}

\medskip

\begin{remark}\label{R:AgmonsChannel}
	It is easy to see that \refT{AnistropicAgmonsInequality} holds as well for a channel domain.
\end{remark}

\medskip

\begin{remark}\label{R:HigherDim}
When we apply \refT{AnistropicAgmonsInequality} in \refS{S-4} we have full control on the tangential derivatives but can control only one derivative in the normal direction. This made the proof of \refT{AnistropicAgmonsInequality} quite simple, as we could apply \refL{Agmon} in 2D to deal with both horizontal derivatives isotropically then use \refL{Agmon} in 1D to deal with the single normal derivative. Had we needed to deal with each variable anisotropically we would have applied \refL{Agmon} in 1D three successive times.
\end{remark}

\medskip

It is worth comparing the inequality in \refT{AnistropicAgmonsInequality} with the 3D Agmon's anisotropic inequality in Proposition 2.2 of \cite{TemamZiane1996}, which can be written, for any $f$ in $H^2(\Omega)$, as
\begin{align}\label{e:TZineq}
	\begin{split}
	\norm{f}_{L^\iny(\Omega)}
		&\ll \norm{f}_{L^2(\Omega)}^{\frac{1}{4}}
			\prod_{j=1}^3 \pr{\norm{f}_{L^2(\Omega)} + \norm{\prt_j f}_{L^2(\Omega)}
				+ \norm{\prt_j \prt_j f}_{L^2(\Omega)}}^{\frac{1}{4}} \\
		&\ll \norm{f}_{L^2(\Omega)}^{\frac{1}{4}}
			\pr{\norm{f}_{L^2(\Omega)} + \norm{\grad f}_{L^2(\Omega)}
				+ \norm{\Delta f}}_{L^2(\Omega)}^{\frac{3}{4}}.
	\end{split}
\end{align}
In \cite{TemamZiane1996}, the authors have some control of the Laplacian but not (directly) of the full $H^2$ norm. This inequality would not work for us, however, as it includes $\prt_3^2 f$.

Both \refT{AnistropicAgmonsInequality} and the inequality in \refE{TZineq} are descendants in spirit of Solonnikov's Theorem 4 of \cite{Sol1972}. The proof in \refE{TZineq} uses, in part, Solonnikov's approach.
The approach we have taken is, however, more elementary and direct than that of \cite{Sol1972}.

Another type of anisotropic inequality that is not a descendant of Solonnikov's theorem (and is not of Agmon type) is the anisotropic embedding inequality of \cite{KTW2010} (Corollary 7.3), which originated in Remark 4.2 of \cite{TW1996OseenIUMJ}, which states that for all $f$ in $H_0^1(\Omega)$,
\begin{align*}
	\norm{f}_{L^\iny(\Omega)}
		\ll \norm{f}_{L^2(\Omega)}^{\frac{1}{2}} \norm{\prt_3 f}_{L^2(\Omega)}^{\frac{1}{2}}
			+ \norm{\prt_1 f}_{L^2(\Omega)}^{\frac{1}{2}}  \norm{\prt_3 f}_{L^2(\Omega)}^{\frac{1}{2}}
			+ \norm{f}_{L^2(\Omega)}^{\frac{1}{2}}  \norm{\prt_1 \prt_ 3 f}_{L^2(\Omega)}^{\frac{1}{2}}.
\end{align*}
Its proof, however, is entirely different from that of \refT{AnistropicAgmonsInequality} or the inequalities in \cite{Sol1972, TemamZiane1996} described above. (A 3D version of it can, however, be obtained using an argument somewhat along the lines of the proof of \refT{AnistropicAgmonsInequality}.)

%
%
\section{Special boundary conditions}\label{A:LionsLikeBCs}

For the Navier-Stokes equations in 2D, when $\al = \kappa$, the Navier boundary conditions reduce to the conditions, $u \cdot \blds{n} = \omega(u) = 0$ (\cite{CMR, FLP, KNavier}).\footnote{In these references, the relation is written $\al = 2 \kappa$, since $2 \Cal{S}(u)$ rather than $\Cal{S}(u)$ is used in the (2D version of) \refE{0-2}.}
Here, $\kappa$ is the curvature of the boundary of a planar bounded domain and $\omega(u)$ is the scalar curl of $u$. The natural extension of this observation to 3D is \refL{LionsLikeBCs}, which involves the shape operator,\footnote{When an inward unit normal convention is used, the expression for $\Cal{A}$ contains a negative sign.}
\begin{align*}
	\Cal{A} v := \pdx{\n}{v} = \grad_v \n,
\end{align*}
the directional derivative of $\blds{n}$ in the direction, $v$, for any vector, $v$, in the tangent plane.

\begin{lemma}\label{L:LionsLikeBCs}
	The boundary conditions in \refE{General_NBC} reduce to those in \refE{LionsLikeBCs} when
	$\Cal{A}$ is the shape operator.
\end{lemma}
\begin{proof}
	Let $\Cal{A}$ be the shape operator and let $\blds{\tau}$ be any unit tangent vector.
	Then since the shape operator is symmetric, we can write \refE{General_NBC} as
	$\mathbf{S}(u^{\ep}) \n \cdot \blds{\tau}   + \mathcal{A} \blds{\tau} \cdot u^{\ep} = 0$.
	But as in \cite{daViegaCrispo2010},
	$
		2 \mathbf{S}(u^\eps)\n \cdot \blds{\tau}
			= (\curl u^\eps \times \n) \cdot \blds{\tau} - 2 u^\eps \cdot \pdx{\n}{\blds{\tau}},
	$
	so \refE{General_NBC} becomes
	\begin{align*}
		(\curl u^\eps \times \n) \cdot \blds{\tau}
			&= 2 \brac{\pdx{\n}{\blds{\tau}} - \mathcal{A} \blds{\tau}} \cdot u^\eps
			= 0.
	\end{align*}
\end{proof}

%
%
\section{Some lemmas}\label{S:SomeLemmas}

\noindent We assume that $\Omega$ is a bounded domain in $\R^3$ having a Lipschitz boundary, $\Gamma$.

\begin{lemma}\label{L:1}
Let $f$ be a divergence-free vector field in $H^2(\Omega)$ that satisfies
\begin{equation*}\label{e:7}
\mathbf{S}(f)  \n  = \Phi \textnormal{ on } \Gamma
\end{equation*}
in the sense of a trace, where $\Phi$ lies in $H^{3/2}(\Gamma)$. Then, for any vector field, $g$, in $H^2$, we have
\begin{equation*}\label{e:8}
-\int_{\Omega} \Delta f \cdot g \, d \blds{x} = 2 \int_{\Omega} \mathbf{S}(f) \cdot \mathbf{S}(g)\,  d \blds{x} -2 \int_\Gamma \Phi \cdot g \, d S,
\end{equation*}
where $A \cdot B = \sum_{1 \leq i, j \leq 3} a_{ij} b_{ij}$ for matrices $A = (a_{ij})_{1 \leq i,j \leq 3}$ and $B = (b_{ij})_{1 \leq i,j \leq 3}$.
\end{lemma}

\smallskip

We recall the following classical lemma:
\begin{lemma}\label{L:trace-like}
Let $u$ be a divergence-free vector field, of class $H^1(\Omega)^3$, in a bounded domain, $\Omega \subset \mathbb{R}^3$, with a $C^2$-boundary, $\Gamma$. Then, if the normal component of $f$ vanishes on $\Gamma$, we have
\begin{equation*}\label{e:trace-like}
|u|_{L^2(\Gamma)} \leq \kap_{\Omega} |u|^{\frac{1}{2}}_{L^2(\Omega)}|\nabla u|^{\frac{1}{2}}_{L^2(\Omega)},
\end{equation*}
for a constant $\kap_{\Omega}$ depending on the domain.
\end{lemma}

\section*{Acknowledgements}

The authors were supported in part by NSF Grants DMS-0842408 and DMS-1009545.
The authors would like to thank Frederick Wilhelm for helpful discussions on the geometry of surfaces, and Drago{\c{s}} Iftimie for pointing out the issue raised in \refR{LipschitzInterpolation}.

\bibliography{Refs}

\def\cprime{$'$} \def\polhk#1{\setbox0=\hbox{#1}{\ooalign{\hidewidth
  \lower1.5ex\hbox{`}\hidewidth\crcr\unhbox0}}}
\begin{thebibliography}{10}

\bibitem{Bat99}
G.~K. Batchelor.
\newblock {\em An introduction to fluid dynamics}.
\newblock Cambridge Mathematical Library. Cambridge University Press,
  Cambridge, paperback edition, 1999.

\bibitem{daViegaCrispo2010}
H.~Beir{\~a}o~da Veiga and F.~Crispo.
\newblock Sharp inviscid limit results under {N}avier type boundary conditions.
  {A}n {$L^p$} theory.
\newblock {\em J. Math. Fluid Mech.}, 12(3):397--411, 2010.

\bibitem{daViegaCrispo2011A}
H.~Beir{\~a}o~da Veiga and F.~Crispo.
\newblock A missed persistence property for the {E}uler equations, and its
  effect on inviscid limits.
\newblock {\em arXiv.org}, 2011.

\bibitem{BelloutNeustupa3}
Hamid Bellout and Ji{\v{r}}{\'{\i}} Neustupa.
\newblock A {N}avier-{S}tokes approximation of the 3{D} {E}uler equation with
  the zero flux on the boundary.
\newblock {\em J. Math. Fluid Mech.}, 10(4):531--553, 2008.

\bibitem{BelloutNeustupa4}
Hamid Bellout, Ji{\v{r}}{\'{\i}} Neustupa, and Patrick Penel.
\newblock On the {N}avier-{S}tokes equation with boundary conditions based on
  vorticity.
\newblock {\em Math. Nachr.}, 269/270:59--72, 2004.

\bibitem{BelloutNeustupa2}
Hamid Bellout, Ji{\v{r}}{\'{\i}} Neustupa, and Patrick Penel.
\newblock On viscosity-continuous solutions of the {E}uler and
  {N}avier-{S}tokes equations with a {N}avier-type boundary condition.
\newblock {\em C. R. Math. Acad. Sci. Paris}, 347(19-20):1141--1146, 2009.

\bibitem{BelloutNeustupa1}
Hamid Bellout, Ji{\v{r}}{\'{\i}} Neustupa, and Patrick Penel.
\newblock On a {$\nu$}-continuous family of strong solutions to the {E}uler or
  {N}avier-{S}tokes equations with the {N}avier-type boundary condition.
\newblock {\em Discrete Contin. Dyn. Syst.}, 27(4):1353--1373, 2010.

\bibitem{Brezis2011}
Haim Brezis.
\newblock {\em Functional analysis, {S}obolev spaces and partial differential
  equations}.
\newblock Universitext. Springer, New York, 2011.

\bibitem{Cia05}
Philippe~G. Ciarlet.
\newblock {\em An introduction to differential geometry with applications to
  elasticity}.
\newblock Springer, Dordrecht, 2005.
\newblock Reprinted from J. Elasticity {{\bf{7}}8/79} (2005), no. 1-3
  [MR2196098].

\bibitem{CMR}
Thierry Clopeau, Andro Mikeli{\'c}, and Raoul Robert.
\newblock On the vanishing viscosity limit for the {$2{\rm D}$} incompressible
  {N}avier-{S}tokes equations with the friction type boundary conditions.
\newblock {\em Nonlinearity}, 11(6):1625--1636, 1998.

\bibitem{Coron1995}
Jean-Michel Coron.
\newblock On the controllability of the {$2$}-{D} incompressible
  {N}avier-{S}tokes equations with the {N}avier slip boundary conditions.
\newblock {\em ESAIM Contr\^ole Optim. Calc. Var.}, 1:35--75 (electronic),
  1995/96.

\bibitem{EW2000}
Weinan E.
\newblock Boundary layer theory and the zero-viscosity limit of the
  {N}avier-{S}tokes equation.
\newblock {\em Acta Math. Sin. (Engl. Ser.)}, 16(2):207--218, 2000.

\bibitem{Eck72}
Wiktor Eckhaus.
\newblock Boundary layers in linear elliptic singular perturbation problems.
\newblock {\em SIAM Rev.}, 14:225--270, 1972.

\bibitem{Gie1}
Gung-Min Gie.
\newblock Singular perturbation problems in a general smooth domain.
\newblock {\em Asymptot. Anal.}, 62(3-4):227--249, 2009.

\bibitem{GHT2}
Gung-Min Gie, Makram Hamouda, and Roger Temam.
\newblock Asymptotic analysis of the {S}tokes problem on general bounded
  domains: the case of a characteristic boundary.
\newblock {\em Appl. Anal.}, 89(1):49--66, 2010.

\bibitem{GHT1}
Gung-Min Gie, Makram Hamouda, and Roger Temam.
\newblock Boundary layers in smooth curvilinear domains: parabolic problems.
\newblock {\em Discrete Contin. Dyn. Syst.}, 26(4):1213--1240, 2010.

\bibitem{Gre98}
Emmanuel Grenier and Olivier Gu{\`e}s.
\newblock Boundary layers for viscous perturbations of noncharacteristic
  quasilinear hyperbolic problems.
\newblock {\em J. Differential Equations}, 143(1):110--146, 1998.

\bibitem{HT1}
Makram Hamouda and Roger Temam.
\newblock Some singular perturbation problems related to the {N}avier-{S}tokes
  equations.
\newblock In {\em Advances in deterministic and stochastic analysis}, pages
  197--227. World Sci. Publ., Hackensack, NJ, 2007.

\bibitem{HT2}
Makram Hamouda and Roger Temam.
\newblock Boundary layers for the {N}avier-{S}tokes equations. {T}he case of a
  characteristic boundary.
\newblock {\em Georgian Math. J.}, 15(3):517--530, 2008.

\bibitem{Hol95}
Mark~H. Holmes.
\newblock {\em Introduction to perturbation methods}, volume~20 of {\em Texts
  in Applied Mathematics}.
\newblock Springer-Verlag, New York, 1995.

\bibitem{IP06}
Drago{\c{s}} Iftimie and Gabriela Planas.
\newblock Inviscid limits for the {N}avier-{S}tokes equations with {N}avier
  friction boundary conditions.
\newblock {\em Nonlinearity}, 19(4):899--918, 2006.

\bibitem{IS10}
Drago{\c{s}} Iftimie and Franck Sueur.
\newblock Viscous boundary layers for the navier-stokes equations with the
  navier slip conditions.
\newblock {\em Arch. Rational Mech. Anal.}, Online First, 20, 2010.

\bibitem{KNavier}
James~P. Kelliher.
\newblock Navier-{S}tokes equations with {N}avier boundary conditions for a
  bounded domain in the plane.
\newblock {\em SIAM Math Analysis}, 38(1):210--232, 2006.

\bibitem{KTW2010}
James~P. Kelliher, Roger Temam, and Xiaoming Wang.
\newblock Boundary layer associated with the {D}arcy-{B}rinkman-{B}oussinesq
  model for convection in porous media.
\newblock {\em Physica D: Nonlinear Phenomena}, 240(7):619--628, 2011.

\bibitem{Klingenberg1978}
Wilhelm Klingenberg.
\newblock {\em A course in differential geometry}.
\newblock Springer-Verlag, New York, 1978.
\newblock Translated from the German by David Hoffman, Graduate Texts in
  Mathematics, Vol. 51.

\bibitem{Koc02}
Herbert Koch.
\newblock Transport and instability for perfect fluids.
\newblock {\em Math. Ann.}, 323(3):491--523, 2002.

\bibitem{Lion73}
J.-L. Lions.
\newblock {\em Perturbations singuli\`eres dans les probl\`emes aux limites et
  en contr\^ole optimal}.
\newblock Lecture Notes in Mathematics, Vol. 323. Springer-Verlag, Berlin,
  1973.

\bibitem{FLP}
M.~C. Lopes~Filho, H.~J. Nussenzveig~Lopes, and G.~Planas.
\newblock On the inviscid limit for 2d incompressible flow with {N}avier
  friction condition.
\newblock {\em SIAM Math Analysis}, 36(4):1130 -- 1141, 2005.

\bibitem{MasmoudiRousset2010}
Nader Masmoudi and Frederic Rousset.
\newblock Uniform regularity for the {N}avier-{S}tokes equation with {N}avier
  friction boundary condition.
\newblock {\em arXiv.org}, 2010.

\bibitem{M1879}
J.~C. Maxwell.
\newblock On stresses in rarified gases arising from inequalities of
  temperature.
\newblock {\em Phil. Trans. Royal Society}, pages 704--712, 1879.

\bibitem{MF53}
Philip~M. Morse and Herman Feshbach.
\newblock {\em Methods of theoretical physics. 2 volumes}.
\newblock McGraw-Hill Book Co., Inc., New York, 1953.

\bibitem{N1827}
C.M.L.H. Navier.
\newblock Sur les lois de l'equilibre et du mouvement des corps \'elastiques.
\newblock {\em Mem. Acad. R. Sci. Inst. France}, 6:369, 1827.

\bibitem{O'M77}
Robert~E. O'Malley, Jr.
\newblock {\em Singular perturbation analysis for ordinary differential
  equations}, volume~5 of {\em Communications of the Mathematical Institute,
  Rijksuniversiteit Utrecht}.
\newblock Rijksuniversiteit Utrecht Mathematical Institute, Utrecht, 1977.

\bibitem{P1905}
L.~Prandtl.
\newblock Verhandlungen des dritten internationalen mathematiker-kongresses in
  heidelberg 1904.
\newblock pages 484--491, 1905.

\bibitem{ShK87}
Shagi-Di Shih and R.~Bruce Kellogg.
\newblock Asymptotic analysis of a singular perturbation problem.
\newblock {\em SIAM J. Math. Anal.}, 18(5):1467--1511, 1987.

\bibitem{Sol1972}
V.~A. Solonnikov.
\newblock Certain inequalities for functions from the classes {$\vec
  W_{p}(R^{n})$}.
\newblock {\em Zap. Nau\v cn. Sem. Leningrad. Otdel. Mat. Inst. Steklov.
  (LOMI)}, 27:194--210, 1972.
\newblock Boundary value problems of mathematical physics and related questions
  in the theory of functions. 6.

\bibitem{TemEuler76}
R.~Temam.
\newblock Local existence of {$C^{\infty }$} solutions of the {E}uler equations
  of incompressible perfect fluids.
\newblock In {\em Turbulence and {N}avier-{S}tokes equations ({P}roc. {C}onf.,
  {U}niv. {P}aris-{S}ud, {O}rsay, 1975)}, pages 184--194. Lecture Notes in
  Math., Vol. 565. Springer, Berlin, 1976.

\bibitem{TW1996OseenIUMJ}
R.~Temam and X.~Wang.
\newblock Asymptotic analysis of {O}seen type equations in a channel at small
  viscosity.
\newblock {\em Indiana Univ. Math. J.}, 45(3):863--916, 1996.

\bibitem{TW02-1}
R.~Temam and X.~Wang.
\newblock Boundary layers associated with incompressible {N}avier-{S}tokes
  equations: the noncharacteristic boundary case.
\newblock {\em J. Differential Equations}, 179(2):647--686, 2002.

\bibitem{TemamZiane1996}
R.~Temam and M.~Ziane.
\newblock Navier-{S}tokes equations in three-dimensional thin domains with
  various boundary conditions.
\newblock {\em Adv. Differential Equations}, 1(4):499--546, 1996.

\bibitem{TemEuler75}
Roger Temam.
\newblock On the {E}uler equations of incompressible perfect fluids.
\newblock {\em J. Functional Analysis}, 20(1):32--43, 1975.

\bibitem{TW95-1}
Roger Temam and Xiao~Ming Wang.
\newblock Asymptotic analysis of the linearized {N}avier-{S}tokes equations in
  a channel.
\newblock {\em Differential Integral Equations}, 8(7):1591--1618, 1995.

\bibitem{TW97-1}
Roger Temam and Xiaoming Wang.
\newblock Asymptotic analysis of the linearized {N}avier-{S}tokes equations in
  a general {$2$}{D} domain.
\newblock {\em Asymptot. Anal.}, 14(4):293--321, 1997.

\bibitem{TW1998}
Roger Temam and Xiaoming Wang.
\newblock On the behavior of the solutions of the {N}avier-{S}tokes equations
  at vanishing viscosity.
\newblock {\em Ann. Scuola Norm. Sup. Pisa Cl. Sci. (4)}, 25(3-4):807--828
  (1998), 1997.
\newblock Dedicated to Ennio De Giorgi.

\bibitem{VL62}
M.~I. Vi{\v{s}}ik and L.~A. Ljusternik.
\newblock Regular degeneration and boundary layer for linear differential
  equations with small parameter.
\newblock {\em Amer. Math. Soc. Transl. (2)}, 20:239--364, 1962.

\bibitem{W1997A}
Xiaoming Wang.
\newblock Time-averaged energy dissipation rate for shear driven flows in
  {$\bold R^n$}.
\newblock {\em Phys. D}, 99(4):555--563, 1997.

\bibitem{XiaoXin2007}
Yuelong Xiao and Zhouping Xin.
\newblock On the vanishing viscosity limit for the 3{D} {N}avier-{S}tokes
  equations with a slip boundary condition.
\newblock {\em Comm. Pure Appl. Math.}, 60(7):1027--1055, 2007.

\end{thebibliography}
\bibliographystyle{plain}

\end{document}